\documentclass[12pt]{article}

\usepackage[T2A]{fontenc}
\usepackage[cp1251]{inputenc}
\usepackage[english]{babel}
\usepackage[tbtags]{amsmath}
\usepackage{amsfonts,amssymb,amsthm,cite}

\usepackage[active]{srcltx}
\usepackage{color}
\usepackage[mathscr]{eucal}
\usepackage{mathrsfs}
\usepackage[sans]{dsfont}

\usepackage{graphicx}

\usepackage[width=165mm, left=23mm, top=20mm, %height=245mm, 
paper=a4paper]{geometry}

\numberwithin{equation}{section}

\newtheorem{theorem}{\hskip\parindent Theorem}[section]
\newtheorem{lemma}{\hskip\parindent Lemma}[section]

\newtheorem{definition}{\hskip\parindent Definition}[section]
\newtheorem{remark}{\hskip\parindent Remark}[section]
\newtheorem{example}{\hskip\parindent Example}[section]

\renewcommand{\proofname}{Proof}

\allowdisplaybreaks

\begin{document}
\begin{center}{\bf\sc	\Large 
Riemann--Liouville integrals in Besov spaces}
%{\Large The study by splines of norm inequalities for Riemann--Liouville operators in weighted Besov spaces}

\smallskip
{Elena P. Ushakova \footnote{elenau@inbox.ru}}

\smallskip
\noindent\textit{\footnotesize V.A. Trapeznikov Institute of Control Sciences of 
Russian Academy of Sciences, 65 Profsoyuznaya str., Moscow 117997, Russia}

\noindent\textit{\footnotesize Steklov Mathematical Institute of Russian Academy of Sciences, 8 Gubkina str., Moscow 119991, Russia}

\noindent\textit{\footnotesize Computing Center of Far Eastern Branch of Russian Academy 
of Sciences, 65 Kim Yu Chena str., Khabarovsk 680000, Russia}
\end{center}

 

\smallskip
{\small {\bf Abstract.}  
Criteria for the fulfillment of inequalities in weighted smoothness function spaces of 
Besov type with Riemann--Liouville operators of natural orders on the real axis and 
semi--axes are found. The obtained estimates are refined under additional conditions on weights.

\section{Introduction}
Consider the left and right Riemann--Liouville operators of natural orders 
$\boldsymbol{\alpha}$ of the forms
\begin{equation}\label{leftL}
I_{c^+}^{\boldsymbol{\alpha}} f({x}):=\frac{1}{\Gamma(\boldsymbol{\alpha})}
\int_{c}^{x}{(x-y)^{\boldsymbol{\alpha}-1}}{f({y})}\,dy\qquad (x>c)
\end{equation} and
\begin{equation}\label{rightR}
I_{c^-}^{\boldsymbol{\alpha}} f({x}):=\frac{1}{\Gamma(\boldsymbol{\alpha})}
\int_{x}^{c}{(x-y)^{\boldsymbol{\alpha}-1}}{f({y})}\,dy\qquad (x<c)
\end{equation} (see e.g. \cite{SKM}). Here $c\in[-\infty,+\infty]$. 
If $c=\mp\infty$ we denote \eqref{leftL} and \eqref{rightR} by $I_{+}^{\boldsymbol{\alpha}}$ and $I_{-}^{\boldsymbol{\alpha}}$, respectively.  

It is assumed that images $I_{c^\pm}^{\boldsymbol{\alpha}}f$ and pre--images $f$ of our operators 
$I_{c^\pm}^{\boldsymbol{\alpha}}$ are identified with density functions
of regular distributions (see e.g. \cite{ST}) from spaces of Besov type $B_{pq}^{s,w}(\mathbb{R})$ with summation parameters
$p>0$ and $0<q\le\infty$, smoothness indices $s\in\mathbb{R}$ and weights $w$
of Muckenhoupt type $\mathscr{A}_\infty^{\rm loc}$ (see related definitions in Section
\ref{Bes}). By weight $w$ we mean a locally integrable almost everywhere positive function
on $\mathbb{R}$. The main problem studied in this paper is characterisation of 
%two--sided 
inequalities connecting the norms of images and pre--images of 
$I_{c^\pm}^{\boldsymbol{\alpha}}$ in $B_{pq}^{s,w}(\mathbb{R})$.

Similar inequalities for some classes of operators acting in Besov type spaces were studied 
in \cite[Theorem 2.3.8]{Tr1},
\cite{N}, \cite[Theorem 2.20]{R}, \cite[\S\,4]{IS}, \cite[p.\,23]{Tr6}. For
$I_{c^\pm}^{\boldsymbol{\alpha}}$ the problem was considered in \cite[Theorem 2.20]{R},
\cite[Proposition 4.2]{IS}, as well as in \cite{{SMJ}, UUaa, Uam} and \cite{Ujms}. In particular, the following results were obtained in
\cite{Ujms}.

\begin{theorem}[{\cite[Theorem 5.5]{Ujms}}]\label{S5T1}
Let
$p> 1$, $0<q\le\infty$, $s\in\mathbb{R}$,
weights $u,v,w\in\mathscr{A}_\infty^{\rm loc}$,
$\boldsymbol{\alpha}\in\mathbb{N}$ and $f\in L^1(\mathbb{R})$.\\
{\rm (i)} Then
$I_{\pm}^{\boldsymbol{\alpha}}f\in {B}_{pq}^{s,u}(\mathbb{R})$
if
$f\in {B}_{pq}^{s+\kappa^*-\boldsymbol{\alpha},v}(\mathbb{R})$ with some
$\kappa^*\in\mathbb{R}$ such that \label{constanty}
\begin{equation}\label{est1}
\mathbf{C}_{\pm}^{\boldsymbol{\alpha}}(\kappa^*)=\mathbf{M}_\pm^{\boldsymbol{\alpha}}(1)+\mathbf{M}_\pm^{\boldsymbol{\alpha}}(0)+
\sup_{d\in\mathbb{N}}\Bigl[\mathscr{M}_\pm^{\boldsymbol{\alpha}}(d,\kappa^*,1)
+\mathscr{M}_\pm^{\boldsymbol{\alpha}}(d,\kappa^*,0)\Bigr]<\infty,\quad\textrm{where}\end{equation}  
\begin{gather*}
\mathbf{M}_{+}^{\theta}(\varepsilon):=\sup_{\tau\in\mathbb{Z}}
\biggl(\sum_{r\ge\tau}(r-\tau+1)^{p(\theta-1)\varepsilon}
u(Q_{0r})\biggr)^{\frac{1}{p}}\biggl(\sum_{r\le\tau} (\tau-r+1)^{p'(\theta-1)(1-\varepsilon)}
[v(Q_{0r})]^{1-p'}\biggr)^{\frac{1}{p'}},\\
\mathbf{M}_{-}^{\theta}(\varepsilon):=\sup_{\tau\in\mathbb{Z}}
\biggl(\sum_{r\le\tau}(\tau-r+1)^{p(\theta-1)\varepsilon}
u(Q_{0r})\biggr)^{\frac{1}{p}}\biggl(\sum_{r\ge\tau} (r-\tau+1)^{p'(\theta-1)(1-\varepsilon)}
[v(Q_{0r})]^{1-p'}\biggr)^{\frac{1}{p'}},\\
\mathscr{M}_{+}^{\theta}(d,\kappa,\varepsilon):={2^{-d\kappa}}\sup_{\tau\in\mathbb{Z}}
\biggl(\sum_{r\ge\tau}(r-\tau+1)^{p(2\theta-1)\varepsilon}
u(Q_{(d-1)r})\biggr)^{\frac{1}{p}}\\\times\biggl(\sum_{r\le\tau} (\tau-r+1)^{p'(2\theta-1)
(1-\varepsilon)}[v(Q_{(d-1)r})]^{1-p'}\biggr)^{\frac{1}{p'}},\\
\mathscr{M}_{-}^{\theta}(d,\kappa,\varepsilon):={2^{-d\kappa}}\sup_{\tau\in\mathbb{Z}}
\biggl(\sum_{r\le\tau}(\tau-r+1)^{p(2\theta-1)\varepsilon}
u(Q_{(d-1)r})\biggr)^{\frac{1}{p}}\\\times\biggl(\sum_{r\ge\tau} (r-\tau+1)^{p'(2\theta-1)
(1-\varepsilon)}[v(Q_{(d-1)r})]^{1-p'}\biggr)^{\frac{1}{p'}}.\end{gather*}
Moreover, 
\begin{equation}\label{4T1more+}
\|I_{\pm}^{\boldsymbol{\alpha}}f\|_{{B}_{pq}^{s,u}(\mathbb{R})}\lesssim \mathbf{C}_{\pm}^{\boldsymbol{\alpha}}(\kappa^*)\|f\|_{{B}_{pq}^{s+\kappa^*-\boldsymbol{\alpha},v}(\mathbb{R})}.
\end{equation}
{\rm (ii)} If
$I_{\pm}^{\boldsymbol{\alpha}}f\in {B}_{pq}^{s,u}(\mathbb{R})$
then
$f\in {B}_{pq}^{s-\kappa_*-\boldsymbol{\alpha},w}(\mathbb{R})$ with
$\kappa_*\in\mathbb{R}$ such that there exist $0\le\varepsilon_M,\varepsilon_\mathbb{M}\le 1$,
satisfying the condition
\begin{equation}\label{est2}\mathbb{C}_{\pm}^{\boldsymbol{\alpha}}(\kappa_*)
={M}_\pm^{\boldsymbol{\alpha}}(\varepsilon_M)+\sup_{d\in\mathbb{N}}
\mathbb{M}_\pm^{\boldsymbol{\alpha}}(d,\kappa_*,\varepsilon_\mathbb{M})<\infty,\quad\textrm{where}\end{equation} 
\begin{gather*}
{M}_{+}^{\theta}(\varepsilon):=\sup_{\tau\in\mathbb{Z}}\biggl(\sum_{r=\tau-\theta}^{\tau}
(\tau-r+1)^{-p(\theta+1)\varepsilon}
w(Q_{0r})\biggr)^{\frac{1}{p}}\biggl(\sum_{r=\tau}^{\tau+\theta} (r-\tau+1)^{-p'(\theta+1)
(1-\varepsilon)}[u(Q_{0r})]^{1-p'}\biggr)^{\frac{1}{p'}},\\
{M}_{-}^{\theta}(\varepsilon):=\sup_{\tau\in\mathbb{Z}}\biggl(\sum_{r=\tau}^{\tau+\theta}
(r-\tau+1)^{-p(\theta+1)\varepsilon}
w(Q_{0r})\biggr)^{\frac{1}{p}}\biggl(\sum_{r=\tau-\theta}^{\tau} (\tau-r+1)^{-p'(\theta+1)
(1-\varepsilon)}[u(Q_{0r})]^{1-p'}\biggr)^{\frac{1}{p'}},\\
\mathbb{M}_{+}^{\theta}(d,\kappa,\varepsilon):={2^{-d\kappa}}\sup_{\tau\in\mathbb{Z}}
\biggl(\sum_{r=\tau-\theta}^{\tau}(\tau-r+1)^{-p(2\theta+1)\varepsilon}
w(Q_{(d-1)r})\biggr)^{\frac{1}{p}}\\\times\biggl(\sum_{r=\tau}^{\tau+\theta} 
(r-\tau+1)^{-p'(2\theta+1)(1-\varepsilon)}[u(Q_{(d-1)r})]^{1-p'}\biggr)^{\frac{1}{p'}},\\
\mathbb{M}_{-}^{\theta}(d,\kappa,\varepsilon):={2^{-d\kappa}}\sup_{\tau\in\mathbb{Z}}
\biggl(\sum_{r=\tau}^{\tau+\theta}(r-\tau+1)^{-p(2\theta+1)\varepsilon}
w(Q_{(d-1)r})\biggr)^{\frac{1}{p}}\\\times\biggl(\sum_{r=\tau-\theta}^{\tau} 
(\tau-r+1)^{-p'(2\theta+1)(1-\varepsilon)}[u(Q_{(d-1)r})]^{1-p'}\biggr)^{\frac{1}{p'}}.
\end{gather*} Besides,
\begin{equation}\label{4T1more-}
\|f\|_{{B}_{pq}^{s-\kappa_*-\boldsymbol{\alpha},w}(\mathbb{R})}\lesssim \mathbb{C}_{\pm}^{\boldsymbol{\alpha}}(\kappa_*)\|I_{\pm}^{\boldsymbol{\alpha}}f\|_{{B}_{pq}^{s,u}(\mathbb{R})}.
\end{equation}
{\rm (iii)} If
$I_{\pm}^{\boldsymbol{\alpha}}f\in {B}_{pq}^{s,u}(\mathbb{R})$
then
$f\in {B}_{pq}^{s-\boldsymbol{\alpha},u}(\mathbb{R})$, and
\begin{equation}\label{est3}
\|f\|_{{B}_{pq}^{s-\boldsymbol{\alpha},u}(\mathbb{R})}\lesssim \|I_{\pm}^{\boldsymbol{\alpha}}f\|_{{B}_{pq}^{s,u}(\mathbb{R})}.
\end{equation}
\end{theorem}

To further list the results from \cite{Ujms}, we introduce the following notations:
\begin{align*}%\label{S6est1}
\widetilde{\mathbf{C}}_{\pm}^{\boldsymbol{\alpha}}(\kappa^*,c)&=
\widetilde{\mathbf{M}}_\pm^{\boldsymbol{\alpha}}(1,c)+\widetilde{\mathbf{M}}_\pm^{\boldsymbol{\alpha}}(0,c)+
\sup_{d\in\mathbb{N}}\Bigl[\widetilde{\mathscr{M}}_\pm^{\boldsymbol{\alpha}}(d,\kappa^*,1,c)
+\widetilde{\mathscr{M}}_\pm^{\boldsymbol{\alpha}}(d,\kappa^*,0,c)\Bigr],
\\\label{S6est2}\widetilde{\mathbb{C}}_{\pm}^{\boldsymbol{\alpha}}(\kappa_*,c)
&=\widetilde{M}_\pm^{\boldsymbol{\alpha}}(\varepsilon_{\widetilde{M}},c)+\sup_{d\in\mathbb{N}}
\widetilde{\mathbb{M}}_\pm^{\boldsymbol{\alpha}}
(d,\kappa_*,\varepsilon_{\widetilde{\mathbb{M}}},c),\end{align*}
where with ${Q_{dr}^{\langle c\rangle }}:=\Bigl[{\frac{r+c}{2^d}},{\frac{r+c+1}{2^d}}\Bigr]$ for
some $c\in\mathbb{R}$
\begin{gather*}
\widetilde{\mathbf{M}}_{+}^{\theta}(\varepsilon,c):=\sup_{\tau\in \mathbb{N}_0}\biggl(\sum_{r\ge\tau}(r-\tau+1)^{p(\theta-1)\varepsilon}
u(Q_{0r}^{\langle c\rangle})\biggr)^{\frac{1}{p}}\biggl(\sum_{0\le r\le\tau} (\tau-r+1)^{p'(\theta-1)(1-\varepsilon)}[v(Q_{0r})^{\langle c\rangle}]^{1-p'}\biggr)^{\frac{1}{p'}},\\
\widetilde{\mathbf{M}}_{-}^{\theta}(\varepsilon,c):=\sup_{\tau\in-\mathbb{N}_0}\biggl(\sum_{r\le\tau}(\tau-r+1)^{p(\theta-1)\varepsilon}
u(Q_{0r}^{\langle c\rangle})\biggr)^{\frac{1}{p}}\biggl(\sum_{0\ge r\ge\tau} (r-\tau+1)^{p'(\theta-1)(1-\varepsilon)}[v(Q_{0r}^{\langle c\rangle})]^{1-p'}\biggr)^{\frac{1}{p'}},\\
\widetilde{\mathscr{M}}_{+}^{\theta}(d,\kappa,\varepsilon,c):={2^{-d\kappa}}\sup_{\tau\in\mathbb{N}_0}\biggl(\sum_{r\ge\tau}(r-\tau+1)^{p(2\theta-1)\varepsilon}
u(Q_{(d-1)r}^{\langle c\rangle})\biggr)^{\frac{1}{p}}\\\times\biggl(\sum_{0\le r\le\tau} (\tau-r+1)^{p'(2\theta-1)(1-\varepsilon)}[v(Q_{(d-1)r}^{\langle c\rangle})]^{1-p'}\biggr)^{\frac{1}{p'}},\\
\widetilde{\mathscr{M}}_{-}^{\theta}(d,\kappa,\varepsilon,c):={2^{-d\kappa}}\sup_{\tau\in-\mathbb{N}_0}\biggl(\sum_{r\le\tau}(\tau-r+1)^{p(2\theta-1)\varepsilon}
u(Q_{(d-1)r}^{\langle c\rangle})\biggr)^{\frac{1}{p}}\\\times\biggl(\sum_{0\ge r\ge\tau} (r-\tau+1)^{p'(2\theta-1)(1-\varepsilon)}
[v(Q_{(d-1)r}^{\langle c\rangle})]^{1-p'}\biggr)^{\frac{1}{p'}},\\
\widetilde{M}_{+}^{\theta}(\varepsilon,c):=\sup_{\tau\in\mathbb{N}_0}\biggl(\sum_{r=\tau-\theta}^{\tau}(\tau-r+1)^{-p(\theta+1)\varepsilon}
w(Q_{0r}^{\langle c\rangle})\biggr)^{\frac{1}{p}}\biggl(\sum_{r=\tau}^{\tau+\theta} (r-\tau+1)^{-p'(\theta+1)(1-\varepsilon)}[u(Q_{0r}^{\langle c\rangle})]^{1-p'}\biggr)^{\frac{1}{p'}},\\
\widetilde{M}_{-}^{\theta}(\varepsilon,c):=\sup_{\tau\in-\mathbb{N}_0}\biggl(\sum_{r=\tau}^{\tau+\theta}(r-\tau+1)^{-p(\theta+1)\varepsilon}
w(Q_{0r}^{\langle c\rangle})\biggr)^{\frac{1}{p}}\biggl(\sum_{r=\tau-\theta}^{\tau} (\tau-r+1)^{-p'(\theta+1)(1-\varepsilon)}[u(Q_{0r}^{\langle c\rangle})]^{1-p'}\biggr)^{\frac{1}{p'}},\\
\widetilde{\mathbb{M}}_{+}^{\theta}(d,\kappa,\varepsilon,c):={2^{-d\kappa}}\sup_{\tau\in\mathbb{N}_0}\biggl(\sum_{r=\tau-\theta}^{\tau}(\tau-r+1)^{-p(2\theta+1)\varepsilon}
w(Q_{(d-1)r}^{\langle c\rangle})\biggr)^{\frac{1}{p}}\\\times\biggl(\sum_{r=\tau}^{\tau+\theta} (r-\tau+1)^{-p'(2\theta+1)(1-\varepsilon)}[u(Q_{(d-1)r}^{\langle c\rangle})]^{1-p'}\biggr)^{\frac{1}{p'}},\\
\widetilde{\mathbb{M}}_{-}^{\theta}(d,\kappa,\varepsilon,c):={2^{-d\kappa}}\sup_{\tau\in-\mathbb{N}_0}\biggl(\sum_{r=\tau}^{\tau+\theta}(r-\tau+1)^{-p(2\theta+1)\varepsilon}
w(Q_{(d-1)r}^{\langle c\rangle})\biggr)^{\frac{1}{p}}\\\times
\biggl(\sum_{r=\tau-\theta}^{\tau} (\tau-r+1)^{-p'(2\theta+1)(1-\varepsilon)}
[u(Q_{(d-1)r}^{\langle c\rangle})]^{1-p'}\biggr)^{\frac{1}{p'}}.\end{gather*}

\begin{theorem}[{\cite[Theorem 6.1]{Ujms}}]\label{S6T1}
Let 
$p> 1$, $0<q\le\infty$, $s\in\mathbb{R}$, 
$u,v\in\mathscr{A}_\infty^{\rm loc}$, 
$\boldsymbol{\alpha}\in\mathbb{N}$ 
and $f\in L_{\rm loc}(\mathbb{R})$. Suppose that
$f\equiv 0$ on $(-\infty,c)$ in the case of operator $I_{c^+}^{\boldsymbol{\alpha}}$ and, similarly, $f\equiv 0$ on $(c,\infty)$ for $I_{c^-}^{\boldsymbol{\alpha}}$.\\ 
{\rm (i)} Then
$I_{c^\pm}^{\boldsymbol{\alpha}}f\in {B}_{pq}^{s,u}(\mathbb{R})$
if
$f\in {B}_{pq}^{s+\kappa^*-\boldsymbol{\alpha},v}(\mathbb{R})$ with some
$\kappa^*\in\mathbb{R}$ such that
$\widetilde{\mathbf{C}}_{\pm}^{\boldsymbol{\alpha}}(\kappa^*,c)<\infty$. 
Moreover, 
\begin{equation}\label{S64T1more+}
\|I_{c^\pm}^{\boldsymbol{\alpha}}f\|_{{B}_{pq}^{s,u}(\mathbb{R})}\lesssim \widetilde{\mathbf{C}}_{\pm}^{\boldsymbol{\alpha}}(\kappa^*,c)\|f\|_{{B}_{pq}^{s+\kappa^*-\boldsymbol{\alpha},v}(\mathbb{R})}.
\end{equation}
{\rm (ii)} If
$I_{c^\pm}^{\boldsymbol{\alpha}}f\in {B}_{pq}^{s,u}(\mathbb{R})$
then
$f\in {B}_{pq}^{s-\kappa_*-\boldsymbol{\alpha},w}(\mathbb{R})$ with such $\kappa_*\in\mathbb{R}$ that with some $0\le\widetilde{\varepsilon_M},\widetilde{\varepsilon_\mathbb{M}}\le 1$
$\widetilde{\mathbb{C}}_{\pm}^{\boldsymbol{\alpha}}(\kappa_*,a)<\infty$.
Besides,
\begin{equation}\label{S64T1more-}
\|f\|_{{B}_{pq}^{s-\kappa_*-\boldsymbol{\alpha},w}(\mathbb{R})}\lesssim \widetilde{\mathbb{C}}_{\pm}^{\boldsymbol{\alpha}}(\kappa_*,a)\|I_{c^\pm}^{\boldsymbol{\alpha}}f\|_{{B}_{pq}^{s,u}(\mathbb{R})}.
\end{equation}
{\rm (iii)} If 
$I_{c^\pm}^{\boldsymbol{\alpha}}f\in {B}_{pq}^{s,u}(\mathbb{R})$ 
then 
$f\in {B}_{pq}^{s-\boldsymbol{\alpha},u}(\mathbb{R})$, and \eqref{est3} is true with $I_{c^\pm}^{\boldsymbol{\alpha}}$ instead of $I_{\pm}^{\boldsymbol{\alpha}}$.
\end{theorem}

In Theorems \ref{S5T1} or \ref{S6T1}, sufficient conditions of the form \eqref{est1}
(or $\widetilde{\mathbf{C}}_{\pm}^{\boldsymbol{\alpha}}(\kappa^*,c)<\infty$)
and \eqref{est2} (or $\widetilde{\mathbb{C}}_{\pm}^{\boldsymbol{\alpha}}(\kappa_*,a)<\infty$)
for the validity of the inequalities \eqref{4T1more+} (or \eqref{S64T1more+}) and \eqref{4T1more-}
(or \eqref{S64T1more-}), respectively, are obtained.
In this paper, we improve these results up to criteria (see Theorems \ref{S5T1-F} 
and \ref{S6T1-F}) by suitably modifying the second component (the mother wavelet) of the spline wavelet systems used for obtaining upper bounds in the decomposition theorems (see Theorem \ref{main*}).

We write $h(\Omega):=\int_\Omega h(x)\, dx$, where {$\Omega \subset\mathbb{R}$}
is a bounded measurable set. Symbol $\mathbb{R}$ denotes the real axis,
$\mathbb{N}$ and $\mathbb{Z}$ --- the sets of natural and integer numbers, respectively. Notation
$\mathbb{N}_0$ is used for the set
$\mathbb{N}\cup\{0\}$, $[s]$ --- for the integer part of $s\in\mathbb{R}$. Relation
$A\lesssim B$ means that $A\le c B$ with some constant $c>0$,
depending only on fixed numerical parameters. Similarly,
$A\simeq B$ is used instead of $A=cB$. We write
$A\approx B$ in the case of $A\lesssim B\lesssim A$. For $r>1$, the conjugate
(or dual) exponent $r'$ is defined as $r/(r-1)$, with $r'=\infty$
if $r=1$. Symbols $:=$ and $=:$ are used to denote new quantities.
By $L_{\rm loc}(\mathbb{R})$ we denote the collection of all locally integrable functions
on $\mathbb{R}$. The notation $L^p(\mathbb{R})$, $0<p\le\infty$, stands for the Lebesgue spaces
of all measurable functions on $\mathbb{R}$ with the (quasi--)norm
$\|f\|_{L^p(\mathbb{R})}:=\bigl(\int_{\mathbb{R}}|f(x)|^pdx\bigr)^{1/p}$ admitting
standard modification for the case $p=\infty$.

\section{Weighted Besov type spaces}\label{Bes}
\subsection{Weights from the Muckenhoupt class ${A}_\infty$}\label{rw} 
Let $Q$ be a segment in $\mathbb{R}$ and $|Q|$ --- its length.  Denote
\begin{gather}\label{Mupa}
{A}_\rho[w(Q)]:=\frac{w(Q)}{|Q|}
\biggl(\frac{1}{|Q|}\int_{Q} w^{1-\rho'}\biggr)^{\frac{\rho}{\rho'}},\qquad {A}_1[w(Q)]:=
\frac{w(Q)}{|Q|}\,\|1/w\|_{L^\infty(Q)}.\end{gather} 

\begin{definition} {\rm (i) Weight $w$ belongs to \textit{Muckenhoupt class} ${A}_\rho$, 
$1<\rho<\infty$, if $$
{A}_\rho(w):=\sup_{Q\subset\mathbb{R}}{A}_\rho[w(Q)]<\infty;$$
(ii) $w\in{A}_1$ if 
 ${A}_1(w):=\sup_{Q\subset\mathbb{R}}{A}_1[w(Q)]<\infty$;\\
(iii) \textit{the Muckenhoupt class} ${A}_\infty$ is $\bigcup_{\rho\ge 1}{A}_\rho$.
}
\end{definition} 

Let us mention some properties of weights from the Muckenhoupt class \cite{M1,S}:\\ 
--- if $w\in{A}_\rho$, $1<\rho<\infty$, then $w^{-\rho'/\rho}\in{A}_{\rho'}$;\\
--- if $1\le \rho_1<\rho_2\le\infty$ then ${A}_{\rho_1}\subset {A}_{\rho_2}$;\\
--- if $w\in{A}_\rho$, $p>1$, then there exists $r<\rho$ such that $w\in{A}_r$.\\ 
The last property implies the existence of the number 
\begin{equation*}\label{r_0}r_0:=r_0(w):= \inf\{r\ge 1\colon w\in{A}_r\}<\infty, 
\qquad w\in{A}_\infty.\end{equation*} 
%играющего роль технического параметра в теоремах декомпозиции.

The functions $w\in A_\infty$ satisfy two important generalised doubling--type conditions \cite{dy}.

(I) For $w\in A_\rho$, $\rho\ge 1$, there exists a constant $c_I>0$ such that for all 
segments $B$ {and} 
$F$ satisfying $F\subset B$, the following holds:
\begin{equation}\label{dpI'}
 \left(\frac{|F|}{|B|}\right)^\rho \le c_I\frac{w(F)}{ w(B)}.
\end{equation}

(II) If $w\in A_\infty$ then there are $\rho^\ast>0$ and $c_{II}>0$ such that for all 
$B$ 
and $E$ such that $E\subset B $, \begin{equation}\label{dpII'}
\frac{w(E)}{ w(B)}\le c_{II}  \left(\frac{|E|}{|B|}\right)^{\rho^\ast}.
\end{equation}

For a more detailed study of the class $\mathscr{A}_\infty$, one can refer to
\cite[Chapter~V]{S} or, for example, to \cite[Lemma 1.3]{HSc}, \cite[Lemma 2.3]{HP}
and \cite[Lemma 1.4]{HSc} (and the literature therein).

\subsection{Weights from the local Muckenhoupt class ${A}_\infty^{\rm loc}$}\label{rwloc} 
The next class of weights covers admissible \cite{HTr2}, 
locally regular \cite{Ma,Sch,Sch'} as well as ${A}_\infty-$weights.

\begin{definition}\label{defin}{\rm (i) Weight $w$ is of 
\textit{local Muckenhoupt type} ${A}_\rho^{\rm loc}$, $1<\rho<\infty$, if (see. \eqref{Mupa}) $$
{A}_\rho^{\rm loc}(w):=\sup_{|Q|\le 1}{A}_\rho[w(Q)]<\infty;$$
(ii) $w\in{A}_1^{\rm loc}$ if $
{A}_1^{\rm loc}(w):=\sup_{|Q|\le 1}{A}_1[w(Q)]<\infty$;\\
(iii) we say that $w\in{A}_\infty^{\rm loc}$ if $w\in {A}_\rho^{\rm loc}$ 
for some $1\le \rho <\infty$, that is $${A}_\infty^{\rm loc}:=\bigcup_{\rho\ge 1}{A}_\rho^{\rm loc}.$$
}\end{definition} 
The weights of class ${A}_\infty^{\rm loc}$ largely repeat the properties of the weights from 
${A}_\infty$ \cite{R}. In particular, by analogy with $r_0$ in the case of $w\in {A}_{\infty}$,
there exists a number
\begin{equation}\label{r_w}{r}_w:= \inf\{r\ge 1\colon w\in{A}_r^{\rm loc}\}<\infty, \qquad w\in{A}_\infty^{\rm loc}.\end{equation} 
The doubling type property for $w\in{A}_\infty^{\rm loc}$ can be formulated as follows:
if $w\in\mathscr{A}_\rho^{\rm loc}$, $1<\rho\le\infty$, then there exists a constant $c_{w}>0$
such that the estimate $$w(Q_t)\le \exp(c_{w}\, t/s) w(Q_s)$$ holds for all $Q_s$ with $|Q_s|=s$,
$0<s\le 1$ and $Q_t$ such that $Q_s\subset Q_t$, where $|Q_t|=t$, $t\ge 1$.

Let us present {\it basic property} of ${A}_\rho^{\rm loc}-$weights with respect to 
${A}_\rho-$functions:
any $w\in{A}_\rho^{\rm loc}$ can be continued outside a given cube $Q$ in such a way that 
a new extended weight
$\bar{w}$ belongs to ${A}_\rho$. 
\begin{lemma}[{\rm \cite[Lemma 1.1]{R}}]\label{basa} 
Let $1\le \rho<\infty$, $w\in{A}_\rho^{\rm loc}$ and $Q$ be a cube with $|Q|=1$. 
Then there is $\bar{w}\in{A}_\rho$ such that $\bar{w}=w$ on $Q$ and with a 
constant $c$ independent of $Q$ the following holds: $${A}_\rho(\bar{w})\le c {A}_\rho^{\rm loc}(w).$$
\end{lemma}
More information about local Muckenhoupt weights can be found in \cite{R,Ma,WBan,WM}.

\subsection{Weighted Besov type spaces} 
Let $p>0$, $0<q\le\infty$ and $s\in\mathbb{R}$. 
 %Символ $F^{-1}\varphi$ будем использовать для обратного преобразования Фурье, заданного аналогичным образом, только с $i$ вместо $-i$ в правой части \eqref{FourierS}. Как $F$, так и $F^{-1}$ можно продолжить на $\mathscr{S}'(\mathbb{R}^N)$.
Definitions of unweighted Besov type spaces $B_{pq}^s(\mathbb{R})$ are given in 
\cite{Tr1, Tr2}. 
For a given $w\not\equiv 1$, their weighted generalisations can be specified in several ways depending on
which class $w$ belongs to \cite{HSc, Rou, Sch, R}. 

The space $\mathscr{S}'(\mathbb{R})$, which is topologically dual to the space 
$\mathscr{S}(\mathbb{R})$ of Schwartz functions, is used in unweighted case, 
as well as in situations when $w$ belongs to the Muckenhoupt class $A_\infty$.
For $\mathscr{A}_\infty^{\rm loc}-$weights, a more general space of distributions 
$\mathscr{S}'_e(\mathbb{R})$ is required than
$\mathscr{S}'(\mathbb{R})$. 
\begin{definition}{\rm 
The space $\mathscr{S}_e(\mathbb{R})$ consists of all functions
${\phi}\in C^\infty(\mathbb{R})$ such that 
$${q}_{N}({\phi}):=\sup_{x\in\mathbb{R}}\mathrm{e}^{{N}|x|}\sum_{0\le 
\gamma\le {N}}|D^\gamma{\phi}(x)|<\infty\qquad\textrm{for any}\quad {N}\in\mathbb{N}_0.$$ }
\end{definition} 

The set $\mathscr{S}_e(\mathbb{R})$ is equipped with a locally convex topology defined 
by the system of seminorms ${q}_{N}$. More details on the properties of 
$\mathscr{S}_e(\mathbb{R})$ can be found in \cite{Sch,R} and \cite[\S~3]{Ma}.
 
By $\mathscr{S}'_e(\mathbb{R})$ we denote the space of all linear continuous
functionals on $\mathscr{S}_e(\mathbb{R})$. The class $\mathscr{S}'_e(\mathbb{R})$
is sometimes conveniently identified with the subset in $\mathscr{D}'(\mathbb{R})$ of all
distributions $f$ on $\mathscr{D}(\mathbb{R})$ for which the estimate
$$|\langle f,{\phi}\rangle|\le C\sup\Bigl\{\bigl|D^\gamma {\phi}(x)\bigr|
\mathrm{e}^{\boldsymbol{N}|x|}\colon x\in\mathbb{R},\,0\le \gamma\le \boldsymbol{N}\Bigr\}$$
is satisfied for any ${\phi}\in\mathscr{D}(\mathbb{R})$ with some constants $C$ and
$\boldsymbol{N}$ depending on $f$. Each such $f$ can be extended to a continuous
functional on $\mathscr{S}_e(\mathbb{R})$.

For ${\phi},{\psi}\in\mathscr{S}_e(\mathbb{R})$, the convolution ${\phi}\ast{\psi}$
belongs to $\mathscr{S}_e(\mathbb{R})$. If $f\in\mathscr{S}'_e(\mathbb{R})$ and
${\phi}\in\mathscr{S}_e(\mathbb{R})$ then the convolution $$f\ast{\phi}(x)=:\langle f(\cdot),
{\phi}(x-\cdot)\rangle,\quad x\in\mathbb{R},$$
is a function in $C^\infty(\mathbb{R})$ of at most exponential growth.

To define Besov spaces with Muckenhoupt weights of local type,
we take ${\phi}_0\in\mathscr{D}(\mathbb{R})$ such that \begin{equation}\label{Iphi1}
\int_{\mathbb{R}}{\phi}_0(x)\,dx\not=0,\end{equation}
further, set
\begin{equation}\label{Iphidef}{\phi}(x)={\phi}_0(x)-2^{-1}{\phi}_0(x/2)\end{equation} and
define ${\phi}_d(x):=2^{(d-1)}{\phi}(2^{d-1}x)$ for $d\in\mathbb{N}$.
One can choose ${\phi}_0$ satisfying \begin{equation}\label{Iphi0}
\int_{\mathbb{R}}x^\gamma{\phi}(x)\,dx=0\end{equation} for all
$\gamma\in\mathbb{N}_0$, $\gamma\le \Gamma$, with $\Gamma\in\mathbb{N}$.
We write $\Gamma=-1$ if \eqref{Iphi0} does not hold. 

%Можно подобрать $\boldsymbol{\varphi}_0$ так, чтобы \begin{equation}\label{Iphi0}\int_{\mathbb{R}^N}x^\gamma\boldsymbol{\varphi}(x)\,dx=0\end{equation} выполнялось для любого мульти--индекса $\gamma\in\mathbb{N}^N_0$, $|\gamma|\le \Gamma$, где $\Gamma\in\mathbb{N}$ фиксировано. Будем писать $\Gamma=-1$, если \eqref{Iphi0} не имеет места.

\begin{definition}\label{SDef2} {\rm 
Let \eqref{Iphi1} hold for some ${\phi}_0\in\mathscr{D}(\mathbb{R})$, and let ${\phi}$
of the form \eqref{Iphidef} satisfy \eqref{Iphi0} with $0\le \gamma\le \Gamma$, where
$\Gamma\ge [s]$. For $w\in A_\infty^{\rm loc}$, we define the weighted Besov type space 
$B_{pq}^{s,w}(\mathbb{R})$ of distributions as the set of all 
$f \in \mathscr{S}'_e(\mathbb{R})$ such that the (quasi--)norm
\begin{equation*}\label{Bspq}
\|f\|_{B_{pq}^{s}(\mathbb{R},w)}: =\biggl(\sum_{d=0}^\infty
2^{dsq}\bigl\| {\phi}_d\ast f\bigr\|_{L^p(\mathbb{R},w)}^q\biggr)^{\frac{1}{q}}
\end{equation*}
(with standard modification for $q=\infty$) is finite.
Here for $p>0$ and a weight $w$ the symbol $L^p(\mathbb{R},w)$ denotes the weighted 
Lebesgue space with the norm
$\|f\|_{L^p(\mathbb{R},w)}:=\|w^{1/p}f\|_{L^p(\mathbb{R})}$.
}
\end{definition}

The definition of $B_{pq}^{s,w}(\mathbb{R})$ does not depend on the choice of the 
function ${\phi}_0$.
It generalises the notion of Besov type smoothness spaces %$B_{pq}^s(\mathbb{R}, w)$
with $w\in A_\infty$ \cite[\S~2.2 and Remark 2.23]{R}.

The spaces $B^{s,w}_{pq}(\mathbb{R})$ are similar in their properties to unweighted
$B^{s}_{pq}(\mathbb{R})$, as well as to the classes $B^{s}_{pq}(\mathbb{R},w)$ with 
$w\in{A}_\infty$
(for more details, see \cite{IS,Ma,R,WBan,WM}).

\smallskip

Next, for $\tau\in\mathbb{Z}$ and $d\in\mathbb{N}_0$ we define dyadic segments
$Q_{d\tau}:=\bigl[{2^{-d}}{\tau},{2^{-d}}{(\tau+1)}\bigr]$ with the origin at the point
$x_{Q_{d\tau}}:=2^{-d}\tau$. In what follows, we will need Besov type spaces 
${b}_{pq}^{s}(w)$ of sequences 
${\boldsymbol{\lambda}}=\{{\boldsymbol{\lambda}}_{d\tau}\}_{d\in\mathbb{N}_0,\,
\tau\in\mathbb{Z}}$, ${\boldsymbol{\lambda}}_{d\tau}\in\mathbb{C}$, such that
\begin{equation*}%\label{bspq}
\|{\boldsymbol{\lambda}}\|_{{b}_{pq}^{s}(w)}: =\Bigl\|\sum_{\tau\in\mathbb{Z}}
|{\boldsymbol{\lambda}}_{0\tau}|\chi_{Q_{0\tau}}\Bigr\|_{L^p(\mathbb{R},w)}+
\biggl(\sum_{d\in\mathbb{N}_0}2^{qd s}\Bigl\|\sum_{\tau\in\mathbb{Z}}|
{\boldsymbol{\lambda}}_{d\tau}|\chi_{Q_{(d-1)\tau}}\Bigr\|_{L^p(\mathbb{R},w)}^q\biggr)^{1/q} <\infty.
\end{equation*}
In the definition of $\|{\boldsymbol{\lambda}}\|_{{b}_{pq}^{s}(w)}$, we can take
$Q_{(d-1)(\tau\pm 1/2)}$ instead of $Q_{(d-1)\tau}$ in the second term.
Due to \eqref{dpI'} property of Muckenhoupt weights (see also Lemma \ref{basa}), 
this is allowed in this situation.

\section{Criteria for the fulfillment of norm inequalities}
The results obtained in this section are based on decomposition theorems those relate
elements $f$ of Besov spaces $B_{pq}^{s,w}(\mathbb{R})$ with
sequential spaces ${b}_{pq}^{s}(w)$, as well as the corresponding to these spaces norms.
To prove them, we decompose $f$ with respect to spline wavelet systems 
of natural orders.

\subsection{Decompositions in terms of spline wavelets}\label{decomp}

Characterization of the spaces $B_{pq}^{s,w}(\mathbb{R})$ in terms of spline wavelets of
natural orders were given in \cite{PSI,Ujms} (see also \cite{RMC}).
In particular, it was established that
\begin{equation}\label{isom}
\|f\|_{B_{pq}^{s,w}(\mathbb{R})}\approx 
\|\widetilde{\boldsymbol{\lambda}}^{\langle\Bbbk,0\rangle}\|_{{b}_{pq}^{s}(w)},\end{equation}
where elements
\begin{equation}\label{asty'}\widetilde{\boldsymbol{\lambda}}_{0\tau}=\langle f,
\widetilde{\mathbf{\Phi}}_{\tau}\rangle\quad(\tau\in\mathbb{Z}),\qquad   
\widetilde{\boldsymbol{\lambda}}_{d\tau}^{\langle\Bbbk,0\rangle}=2^{d/2}\langle f,
\widetilde{\mathbf{\Psi}}^{\langle\Bbbk,0\rangle}_{(d-1)\tau}\rangle\quad(d\in\mathbb{N},
\,\tau\in\mathbb{Z})\end{equation}
of sequences $\widetilde{\boldsymbol{\lambda}}^{\langle\Bbbk,0\rangle}
=\{\widetilde{\boldsymbol{\lambda}}_{0,\tau}\}_{\tau\in\mathbb{Z}}
\cup\{\widetilde{\boldsymbol{\lambda}}^{\langle\Bbbk,0\rangle}_{d\tau}\}_{d\in\mathbb{N},
\,\tau\in\mathbb{Z}}$ с $\Bbbk\in\{0,1\}$ are defined by $\widetilde{\mathbf{\Phi}}_{\tau}$ и 
$\widetilde{\mathbf{\Psi}}^{\langle\Bbbk,0\rangle}_{(d-1)\tau}$ of the form \eqref{ForRepr'} 
with possible "starting"\, points $\boldsymbol{a}\in\{0,\pm1/2\}$. 
In \cite{PSI}, the case $\boldsymbol{a}=0$ was considered, while in \cite{RMC} ---
an even simpler situation $\boldsymbol{a}=\Bbbk=0$. 

The result \eqref{isom}
allows for some generalization (see Theorem \ref{main*}). Let us start, however,
with the base case $\boldsymbol{\zeta}=\Bbbk=0$. For this, taking into account
\eqref{r_w}, we denote 
$$\sigma_{p}(w):={{r}_w}/{\min\{p,{r}_w\}}-2+{r}_w.$$ %Будем считать, что точка $\boldsymbol{a}$ находится от нуля 
%на расстоянии, сравнимом с $n+m$ (см. определения $\Phi_{n,\boldsymbol{a}}=
%B_{n,\boldsymbol{a}}$ и ${\Psi}_{{n,\boldsymbol{a}},m(\Bbbk)}$ в \S\, 5). 
We will assume that $\boldsymbol{s}=0$ in the definitions of
${\mathbf{\Psi}}^{\langle\boldsymbol{a}\rangle}\in \eqref{vazhnooW}$ and
$\widetilde{\mathbf{\Psi}}^{\langle \Bbbk,\boldsymbol{\zeta}\rangle}\in\eqref{vazhnoo}$,
$\Bbbk,\boldsymbol{\zeta}\in\{0,1\}$.

\begin{theorem}[\cite{Ujms,PSI}]\label{main'}
Let $p>0$, $0<q\le\infty$, $s\in\mathbb{R}$ and $w\in{A}_\infty^{\rm loc}$.
Let ${\mathbf{\Phi}}^{\langle\boldsymbol{a}\rangle}$ and
${\mathbf{\Psi}}^{\langle\boldsymbol{a}\rangle}$, where $\boldsymbol{a}\in\{0,\pm 1/2\}$,
be functions satisfying \eqref{vazhnooW} with some $n\in\mathbb{N}$ such that 
\begin{equation}\label{condB}n\ge \max\Bigl\{0,[s]+1,{\bigr[({r}_w-1)/p-s\bigr]}+1,
\bigl[\sigma_{p}(w)-s\bigr]\Bigr\}+1.\end{equation} 
Then $f\in\mathscr{S}'_e(\mathbb{R})$ belongs to $B_{pq}^{s,w}(\mathbb{R})$ 
if and only if the representation holds
\begin{equation}\label{function}
f=\sum_{\tau\in\mathbb{Z}} \boldsymbol{\lambda}_{0\tau}^{\langle\boldsymbol{a}\rangle}
{\mathbf{\Phi}}^{\langle\boldsymbol{a}\rangle}_\tau+
\sum_{d\in\mathbb{N}}\sum_{\tau\in\mathbb{Z}}\boldsymbol{\lambda}_{d\tau}^{\langle
\boldsymbol{a}\rangle}2^{-d/2}{\mathbf{\Psi}}_{(d-1)\tau}^{\langle\boldsymbol{a}\rangle},
\end{equation} where $\{{\boldsymbol{\lambda}}^{\langle
\boldsymbol{a}\rangle}_{d\tau}\}_{d\in\mathbb{N}_0,\,\tau\in\mathbb{Z}}=:
\boldsymbol{\lambda}^{\langle\boldsymbol{a}\rangle}\in b_{pq}^{s}(w)$ and the series
converges in
$\mathscr{S}'_e(\mathbb{R})$. The representation \eqref{function} is unique with
\begin{equation}\label{asty}\boldsymbol{\lambda}_{0\tau}^{\langle\boldsymbol{a}\rangle}=
\langle f,{\mathbf{\Phi}}_{\tau}^{\langle\boldsymbol{a}\rangle}\rangle\quad(\tau\in\mathbb{Z}),\qquad   \boldsymbol{\lambda}_{d\tau}^{\langle\boldsymbol{a}\rangle}=2^{d/2}\langle f,{\mathbf{\Psi}}_{(d-1)\tau}^{\langle\boldsymbol{a}\rangle}\rangle\quad(d\in\mathbb{N},\,\tau\in\mathbb{Z}),\end{equation}
and $I\colon f\mapsto \boldsymbol{\lambda}^{\langle\boldsymbol{a}\rangle}$ is a linear isomorphism of
$B_{pq}^{s,w}(\mathbb{R})$ onto $b_{pq}^{s}(w)$. Moreover, \begin{equation}\label{isom1}
\|f\|_{B_{pq}^{s,w}(\mathbb{R})}\approx \|{\boldsymbol{\lambda}}^{\langle\boldsymbol{a}\rangle}
\|_{{b}_{pq}^{s}(w)},\end{equation}
where the equivalence constants depend only on the numerical parameters. %, в частности от $p$, $q$, $r_w$ {\rm(}см. \eqref{dpI'}{\rm)}, $|\boldsymbol{a}|$ и $|\boldsymbol{s}|$.
\end{theorem}

Below there is a formulation of an analogue of Theorem \ref{main'}, which uses
a dictionary of functions
$\widetilde{\mathbf{\Phi}}$ and $\widetilde{\mathbf{\Psi}}^{\langle\Bbbk,0
\rangle}$ of the form
\eqref{vazhnoo} with $\Bbbk=1$ and $\boldsymbol{\zeta}=0$. It also
asserts \eqref{isom} (see \cite[Theorem~4.8]{Ujms}), but with different representations %\eqref{function}
for $f$ in its sufficiency and necessity parts.

\begin{theorem}\label{main}
Let $p>0$, $0<q\le\infty$, $s\in\mathbb{R}$ and $w\in{A}_\infty^{\rm loc}$. Suppose that
for some $n\in\mathbb{N}$ the condition \eqref{condB} is satisfied, and there are functions
${\mathbf{\Phi}}^{\langle\boldsymbol{a}\rangle}$ and
${\mathbf{\Psi}}^{\langle\boldsymbol{a}\rangle}$ with $\boldsymbol{a}\in\{0,\pm 1/2\}$,
satisfying \eqref{vazhnooW}, as well as functions $\widetilde{\mathbf{\Phi}}$ and
$\widetilde{\mathbf{\Psi}}^{\langle 1,0\rangle}$ of the form \eqref{vazhnoo}, 
both the pairs of smoothness $n$.

{\rm (i)} A distribution $f\in\mathscr{S}'_e(\mathbb{R})$ belongs to 
$B_{pq}^{s,w}(\mathbb{R})$, if there is a representation
\begin{equation}\label{function'}
f=\sum_{\tau\in\mathbb{Z}} \widetilde{\boldsymbol{\lambda}}_{0\tau}
\widetilde{\mathbf{\Phi}}_\tau+
\sum_{d\in\mathbb{N}}\sum_{\tau\in\mathbb{Z}}
\widetilde{\boldsymbol{\lambda}}_{d\tau}^{\langle 1,0\rangle} 2^{-d/2}
\widetilde{\mathbf{\Psi}}_{(d-1)\tau}^{\langle 1,0\rangle},
\end{equation} where $\widetilde{\boldsymbol{\lambda}}^{\langle 1,0\rangle}$ of the form 
\eqref{asty'} is from $b_{pq}^{s}(w)$ and the series \eqref{function'} 
converges in $\mathscr{S}'_e(\mathbb{R})$. 
%Представление \eqref{function'} верно с
%\begin{equation}\label{asty'}\widetilde{\lambda}_{0\tau}=\langle f,\widetilde{\mathbf{\Phi}}_{\tau}\rangle\quad(\tau\in\mathbb{Z}),\qquad   \widetilde{\lambda}_{d\tau}=2^{d/2}\langle f,\widetilde{\mathbf{\Psi}}_{(d-1)\tau}\rangle\quad(d\in\mathbb{N},\,\tau\in\mathbb{Z}).\end{equation}
Moreover, %имеет место оценка 
\begin{equation*}\label{isom'}
\|f\|_{B_{pq}^{s,w}(\mathbb{R})}\lesssim \|
\widetilde{\boldsymbol{\lambda}}^{\langle 1,0\rangle}\|_{{b}_{pq}^{s}(w)}\lesssim 
\sum_{\boldsymbol{a}=0,\pm1/2}\|{\boldsymbol{\lambda}}^{\langle 
\boldsymbol{a}\rangle}\|_{{b}_{pq}^{s}(w)}\end{equation*} 
with ${\boldsymbol{\lambda}}^{\langle \boldsymbol{a}\rangle}$ of the form \eqref{asty}.\\
{\rm (ii)} Conversely, if $f\in\mathscr{S}'_e(\mathbb{R})$ belongs to
$B_{pq}^{s,w}(\mathbb{R})$, then for $\boldsymbol{a}\in\{0,\pm1/2\}$ we have
\eqref{function} in the sense of convergence in $\mathscr{S}'_e(\mathbb{R})$,
where $\boldsymbol{\lambda}^{\langle \boldsymbol{a}\rangle}\in b_{pq}^{s}(w)$
is given by \eqref{asty}. Moreover, 
\begin{equation*}\label{isomN}
\|\widetilde{\boldsymbol{\lambda}}^{\langle 1,0\rangle}\|_{{b}_{pq}^{s}(w)}\le 
\sum_{\boldsymbol{a}=0,1/2}\|{\boldsymbol{\lambda}}^{\langle 
\boldsymbol{a}\rangle}\|_{{b}_{pq}^{s}(w)}\lesssim\|f\|_{B_{pq}^{s,w}(\mathbb{R})}.
\end{equation*}
%В оценках \eqref{isom'} и \eqref{isomN}, константы эквивалентности зависят от $r_w$ и $|\boldsymbol{a}|$ {\rm(}см. \eqref{dpI'}{\rm)}.
\end{theorem}

A generalisation of Theorem \ref{main} in terms of
$\widetilde{\mathbf{\Phi}}$ and $\widetilde{\mathbf{\Psi}}^{\langle1,1
\rangle}$ of the form \eqref{vazhnoo} is given below.

\begin{theorem}\label{main*}
Let $p>0$, $0<q\le\infty$, $s\in\mathbb{R}$ and $w\in{A}_\infty^{\rm loc}$. Suppose that
for some $n\in\mathbb{N}$ the condition \eqref{condB} is satisfied, and there are functions
${\mathbf{\Phi}}^{\langle\boldsymbol{a}\rangle}$ and
${\mathbf{\Psi}}^{\langle\boldsymbol{a}\rangle}$ with $\boldsymbol{a}\in\{0,\pm 1/2\}$,
satisfying \eqref{vazhnooW}, as well as functions $\widetilde{\mathbf{\Phi}}$ and
$\widetilde{\mathbf{\Psi}}^{\langle 1,1\rangle}$ of the form \eqref{vazhnoo}, 
the both pairs of smoothness $n$.

{\rm (i)} A distribution $f\in\mathscr{S}'_e(\mathbb{R})$ belongs to 
$B_{pq}^{s,w}(\mathbb{R})$ if there is a representation
\begin{equation*}\label{function'*}
f=\sum_{\tau\in\mathbb{Z}} \widetilde{\boldsymbol{\lambda}}_{0\tau}
\widetilde{\mathbf{\Phi}}_\tau+
\sum_{d\in\mathbb{N}}\sum_{\tau\in\mathbb{Z}}
\widetilde{\boldsymbol{\lambda}}_{d\tau}^{\langle 1,1\rangle} 2^{-d/2}
\widetilde{\mathbf{\Psi}}_{(d-1)\tau}^{\langle 1,1\rangle},
\end{equation*} 
where the sequence $\widetilde{\boldsymbol{\lambda}}^{\langle1,1\rangle}
=\{\widetilde{\boldsymbol{\lambda}}_{0,\tau}\}_{\tau\in\mathbb{Z}}
\cup\{\widetilde{\boldsymbol{\lambda}}^{\langle1,1\rangle}_{d\tau}\}_{d\in\mathbb{N},
\,\tau\in\mathbb{Z}}$ of the form
\begin{equation}\label{asty'*}\widetilde{\boldsymbol{\lambda}}_{0\tau}=\langle f,
\widetilde{\mathbf{\Phi}}_{\tau}\rangle\quad(\tau\in\mathbb{Z}),\qquad   
\widetilde{\boldsymbol{\lambda}}_{d\tau}^{\langle1,1\rangle}=2^{d/2}\langle f,
\widetilde{\mathbf{\Psi}}^{\langle1,1\rangle}_{(d-1)\tau}\rangle\quad(d\in\mathbb{N},
\,\tau\in\mathbb{Z})\end{equation} is from $b_{pq}^{s}(w)$ with 
$\widetilde{\mathbf{\Phi}}_{\tau}$ and 
$\widetilde{\mathbf{\Psi}}^{\langle1,1\rangle}_{(d-1)\tau}$ given by \eqref{ForRepr'} with 
$\boldsymbol{a}=\{0,\pm1/2\}$. In this case, the series
\eqref{function'} converges in $\mathscr{S}'_e(\mathbb{R})$. In addition, with
${\boldsymbol{\lambda}}^{\langle \boldsymbol{a}\rangle}$
of the form \eqref{asty}, we have
\begin{equation}\label{isom'*}
\|f\|_{B_{pq}^{s,w}(\mathbb{R})}\lesssim \|
\widetilde{\boldsymbol{\lambda}}^{\langle 1,1\rangle}\|_{{b}_{pq}^{s}(w)}\lesssim
\sum_{\boldsymbol{a}=0,\pm1/2}\|{\boldsymbol{\lambda}}^{\langle \boldsymbol{a}\rangle}
\|_{{b}_{pq}^{s}(w)}.\end{equation}
{\rm (ii)} Conversely, if $f\in\mathscr{S}'_e(\mathbb{R})$ belongs to
$B_{pq}^{s,w}(\mathbb{R})$, then the
\eqref{function} holds for $\boldsymbol{a}\in\{0,\pm1/2\}$ in the sense of convergence in 
$\mathscr{S}'_e(\mathbb{R})$, where
$\boldsymbol{\lambda}^{\langle \boldsymbol{a}\rangle}\in b_{pq}^{s}(w)$
is given by \eqref{asty}. Besides,
\begin{equation*}\label{isomN*}
\|\widetilde{\boldsymbol{\lambda}}^{\langle 1,1\rangle}\|_{{b}_{pq}^{s}(w)}\le
\sum_{\boldsymbol{a}=0.1/2}\|{\boldsymbol{\lambda}}^{\langle \boldsymbol{a}\rangle}\|
_{{b}_{pq}^{s}(w)}\lesssim\|f\|_{B_{pq}^{s,w}(\mathbb{R})}.\end{equation*}
%В оценках \eqref{isom'} и \eqref{isomN}, константы эквивалентности зависят от $r_w$ и $|\boldsymbol{a}|$ {\rm(}см. \eqref{dpI'}{\rm)}.
\end{theorem}

Observe that in the original decomposition theorem
using spline wavelet systems of natural orders (see \cite[\S\,4.3]{PSI}) the case
$\boldsymbol{a}=0$ was considered. Later 
in \cite[Theorem~4.8]{Ujms} the cases $\boldsymbol{a}=\pm 1/2$ were added, they are 
justified by the \eqref{dpI'} property of Muckenhoupt weights and Lemma \ref{basa}.
The proof of Theorem \ref{main*} is similar to the proof of Theorem \ref{main}
(see \cite[Theorem~4.8]{Ujms}, as well as \cite[Theorem~4.3]{PSI}). 
The basis for it is the fact that the elements of the sequence 
$\widetilde{\boldsymbol{\lambda}}^{\langle1,1\rangle}$
form a family of atoms \cite{WBan, IS} (see also \cite[Definition~4.1]{PSI}).

\subsection{The main result}
Let us introduce into consideration a functional
\begin{equation*}
\mathfrak{M}(d,\kappa;\sigma_1,\sigma_2):=\sup_{\tau\in\mathbb{Z}}
\begin{cases}
\bigl[\sigma_1(Q_{0\tau})\bigr]^{-\frac{1}{p}}\bigl[\sigma_2(Q_{0\tau})\bigr]^{\frac{1}{p}}, &d=0,\\
2^{-d\kappa}\bigl[\sigma_1(Q_{(d-1)\tau})\bigr]^{-\frac{1}{p}}
\bigl[\sigma_2(Q_{(d-1)\tau})\bigr]^{\frac{1}{p}}, & d\in\mathbb{N}.
\end{cases}\end{equation*}

\begin{theorem}\label{S5T1-F}
Let
$p> 1$, $0<q\le\infty$, $s\in\mathbb{R}$, weights
$u,v\in\mathscr{A}_\infty^{\rm loc}$ and
$\boldsymbol{\alpha}\in\mathbb{N}$. Suppose that $f\in L^1(-\infty,x)$ for any
$x\in\mathbb{R}$ in the case of $I_{+}^{\boldsymbol{\alpha}}$ and $f\in L^1(x,+\infty)$ for all
$x\in\mathbb{R}$ for $I_{-}^{\boldsymbol{\alpha}}$.\\ 
{\rm (i)} The operator $I_{\pm}^{\boldsymbol{\alpha}}$ is bounded from
${B}_{pq}^{s+\kappa^*-\boldsymbol{\alpha},v}(\mathbb{R})$ to
${B}_{pq}^{s,u}(\mathbb{R})$ with some $\kappa^*\in\mathbb{R}$
if and only if
\begin{equation}\label{estN1}
\mathfrak{C}_{\pm}^{\boldsymbol{\alpha}}(\kappa^*)=
\mathbf{M}_\pm^{\boldsymbol{\alpha}}(1)+\mathbf{M}_\pm^{\boldsymbol{\alpha}}(0)+
\sup_{d\in\mathbb{N}}\mathfrak{M}(d,\kappa^*;v,u)<\infty.\end{equation}
In addition, the norm $\|I_{\pm}^{\boldsymbol{\alpha}}\|_{{B}_{pq}^{s+\kappa^*-\boldsymbol{\alpha},v}(\mathbb{R})\to
{B}_{pq}^{s,u}(\mathbb{R})}$ of the operator $I_{\pm}^{\boldsymbol{\alpha}}
\colon {B}_{pq}^{s+\kappa^*-\boldsymbol{\alpha},v}(\mathbb{R})\to {B}_{pq}^{s,u}(\mathbb{R})$
is equivalent to $\mathfrak{C}_{\pm}^{\boldsymbol{\alpha}}(\kappa^*)$ with constants
depending only on the numerical parameters.
\\
{\rm (ii)} Let $\kappa_*\in\mathbb{R}$ and suppose that $0<\mathds{C}$ is
the best constant of the inequality
\begin{equation}\label{4T1more-_F}
\|f\|_{{B}_{pq}^{s-\kappa_*-\boldsymbol{\alpha},w}(\mathbb{R})}\lesssim
\mathds{C}\|I_{\pm}^{\boldsymbol{\alpha}}f\|_{{B}_{pq}^{s,u}(\mathbb{R})}
\end{equation} independent of $f$. Inequality \eqref{4T1more-_F} holds if and only if
$$\mathds{C}(\kappa_*):=\sup_{d\in\mathbb{N}_0}\mathfrak{M}(d,\kappa_*;u,w)<\infty,$$ and
$\mathds{C}\approx \mathds{C}(\kappa_*)$.
In particular, if $w=u$, then
$\mathds{C}\approx \mathds{C}(0)=1$.
\end{theorem}
\renewcommand{\proofname}{Proof}
\begin{proof} (i) {\it Sufficiency.}
Let us turn to the beginning of the proof of assertion (i) from 
\cite[Theorem 5.5]{Ujms} in the sufficiency part and show the validity of
$$\|I_{+}^{\boldsymbol{\alpha}}\|_{{B}_{pq}^{s+
\kappa^*-\boldsymbol{\alpha},v}(\mathbb{R})\to
{B}_{pq}^{s,u}(\mathbb{R})}\lesssim \mathfrak{C}_{+}^{\boldsymbol{\alpha}}(\kappa^*).$$
To do this, we choose $n^\ast\in\mathbb{N}$ satisfying the condition \eqref{condB} with respect
to parameters of the space $B^{s,u}_{pq}(\mathbb{R})$ including $r_u$ (see \eqref{r_w}). Besides,
we set $m^\ast:= n^\ast+\boldsymbol{\alpha}$, which must also satisfy \eqref{condB} 
with respect to parameters of the space 
$B^{s+\kappa^*-\boldsymbol{\alpha},v}_{pq}( \mathbb{R})$ and $r_v$.
Then, by \eqref{isom'*} (see Theorem \ref{main*}(i)),
$$
\|I_{+}^{\boldsymbol{\alpha}}f\|_{B^{s,u}_{pq}(\mathbb{R})}\approx
\|\widetilde{\boldsymbol{\lambda}}^{\langle 1,1\rangle}\|_{b_{pq}^{s}(u)},$$
where
$\widetilde{\boldsymbol{\lambda}}^{\langle 1,1\rangle}$
is defined by \eqref{asty'*} with $I_{+}^{\boldsymbol{\alpha}}f$ instead of $f$. 
Besides, $\widetilde{\boldsymbol{\Phi}}$ is given by $B_{n*}$, 
and $\widetilde{\boldsymbol{\Psi}}^{\langle 1,1\rangle}$ (see \eqref{vazhnoo}) 
is specified by $\Psi_{n^\ast,0,0;m^\ast(1),\boldsymbol{\alpha}(1)}$.
This means that we can write:
\begin{equation}\label{rysha+U}\|I_{+}^{\boldsymbol{\alpha}}f\|_{B^{s,u}_{pq}(\mathbb{R})}
\approx 
\Bigl\|\sum_{\tau\in\mathbb{Z}}
|\widetilde{\boldsymbol{\lambda}}_{0\tau}|\chi_{Q_{0\tau}}\Bigr\|_{L^p(\mathbb{R},u)}+
\biggl(\sum_{d\in\mathbb{N}}2^{qd s}\Bigl\|\sum_{\tau\in\mathbb{Z}}
|\widetilde{\boldsymbol{\lambda}}_{d\tau}^{\langle1,1\rangle}|\chi_{Q_{(d-1)\tau}}
\Bigr\|_{L^p(\mathbb{R},u)}^q\biggr)^{1/q},\end{equation} where
\begin{equation*}%\label{asty'}
\widetilde{\boldsymbol{\lambda}}_{0\tau}=\langle I_{+}^{\boldsymbol{\alpha}}f,
\widetilde{\mathbf{\Phi}}_{\tau}\rangle\quad(\tau\in\mathbb{Z}),\qquad   
\widetilde{\boldsymbol{\lambda}}_{d\tau}^{\langle1,1\rangle}=2^{d/2}\langle 
I_{+}^{\boldsymbol{\alpha}}f,\widetilde{\mathbf{\Psi}}^{\langle1,1\rangle}_{(d-1)\tau}
\rangle\quad(d\in\mathbb{N},\,\tau\in\mathbb{Z}).\end{equation*}

The first term on the right--hand side of \eqref{rysha+U} can be estimated by repeating
a similar stage of the proof of assertion (i) from \cite[Theorem 5.5]{Ujms}.
We present it here in a slightly modified form, by relying on
the differential property \eqref{diff} of the basic splines $B_n$. 

As in the proof of \cite[Theorem 5.5]{Ujms}, we denote $\widetilde{\mathbf{\Phi}}_\tau(\cdot)=
B_{n^*}(\cdot-\tau)=:P(\tau)$ and set  
\begin{equation*} %\label{RP}
R(\tau-r):=\sum_{j=0}^{\boldsymbol{\alpha}}(-1)^j
\binom{\boldsymbol{\alpha}}{j}P(\tau-r-j).
\end{equation*}
It is true that
\begin{equation}\label{RP1}
P(\tau-r)=R(\tau-r)-\sum_{j=1}^{\boldsymbol{\alpha}}(-1)^j\binom{\boldsymbol{\alpha}}{j}P(\tau-r-j).
\end{equation} Then
\begin{equation*}%\label{RP1}
P(\tau)=R(\tau)-\sum_{j=1}^{\boldsymbol{\alpha}}(-1)^j\binom{\boldsymbol{\alpha}}{j}P(\tau-j)=R(\tau)+
\binom{\boldsymbol{\alpha}}{1}P(\tau-1)- \sum_{j=2}^{\boldsymbol{\alpha}}(-1)^j
\binom{\boldsymbol{\alpha}}{j}P(\tau-j).
\end{equation*} Again on the strength of \eqref{RP1},
\begin{multline*}%\label{RP1}
P(\tau)-R(\tau)=\binom{\boldsymbol{\alpha}}{1}\biggl[R(\tau-1)-\sum_{j=1}^{\boldsymbol{\alpha}}(-1)^j
\binom{\boldsymbol{\alpha}}{j}P(\tau-1-j)\biggr]-\sum_{j=2}^{\boldsymbol{\alpha}}(-1)^j
\binom{\boldsymbol{\alpha}}{j}P(\tau-j)\\=
\binom{\boldsymbol{\alpha}}{1}R(\tau-1)-\binom{\boldsymbol{\alpha}}{1}
\sum_{j=1}^{\boldsymbol{\alpha}}(-1)^j\binom{\boldsymbol{\alpha}}{j}P(\tau-1-j)-
\sum_{j=2}^{\boldsymbol{\alpha}}(-1)^j\binom{\boldsymbol{\alpha}}{j}P(\tau-j)=
\binom{\boldsymbol{\alpha}}{1}R(\tau-1)\\+\biggl[{\binom{\boldsymbol{\alpha}}{1}}^2-
\binom{\boldsymbol{\alpha}}{2}\biggr]P(\tau-2)
-\binom{\boldsymbol{\alpha}}{1}\sum_{j=2}^{\boldsymbol{\alpha}}(-1)^j\binom{\boldsymbol{\alpha}}{j}P(\tau-1-j)-
\sum_{j=3}^{\boldsymbol{\alpha}}(-1)^j\binom{\boldsymbol{\alpha}}{j}P(\tau-j).
\end{multline*} Setting $A_0:=\binom{\boldsymbol{\alpha}}{0}=1$, $A_1:=A_0
\binom{\boldsymbol{\alpha}}{1}$, $A_2:=A_1\binom{\boldsymbol{ \alpha}}{1}-
A_0\binom{\boldsymbol{\alpha}}{2}
=\binom{\boldsymbol{\alpha}}{1}^2-\binom{\boldsymbol{\alpha}}{2}$,
we continue the procedure for forming finite iterated differences as shown below
\begin{multline*}%\label{RP1}
P(\tau)=A_0R(\tau)+A_1R(\tau-1)+A_2\biggl[R(\tau-2)-\sum_{j=1}^{\boldsymbol{\alpha}}(-1)^j
\binom{\boldsymbol{\alpha}}{j}P(\tau-2-j)\biggr]\\
-\binom{\boldsymbol{\alpha}}{1}\sum_{j=2}^{\boldsymbol{\alpha}}(-1)^j
\binom{\boldsymbol{\alpha}}{j}P(\tau-1-j)-\sum_{j=3}^{\boldsymbol{\alpha}}(-1)^j
\binom{\boldsymbol{\alpha}}{j}P(\tau-j)\\
=A_0R(\tau)+A_1R(\tau-1)+A_2R(\tau-2)-\biggl[\underbrace{A_2\binom{\boldsymbol{\alpha}}{1}-A_1
\binom{\boldsymbol{\alpha}}{2}+A_0\binom{\boldsymbol{\alpha}}{3}}_{:=A_3}\biggr]P(\tau-3)\\
-A_2\sum_{j=2}^{\boldsymbol{\alpha}}(-1)^j\binom{\boldsymbol{\alpha}}{j}P(\tau-2-j)
-\binom{\boldsymbol{\alpha}}{1}\sum_{j=3}^{\boldsymbol{\alpha}}(-1)^j
\binom{\boldsymbol{\alpha}}{j}P(\tau-1-j)-\sum_{j=4}^{\boldsymbol{\alpha}}(-1)^j
\binom{\boldsymbol{\alpha}}{j}P(\tau-j),
\end{multline*} which is resulting in the sum of the form
\begin{equation*}%\label{RP1}
P(\tau)=\sum_{r\ge 0} A_rR(\tau-r),\qquad \textrm{where}\quad A_r=A_r(\boldsymbol{\alpha})=
\sum_{j=1}^r(-1)^{j-1}A_{r-j}\binom{\boldsymbol{\alpha}}{j}.\end{equation*}
In \cite{Ujms}, it was established that
$A_r(\boldsymbol{\alpha})\le (1+r)^{\boldsymbol{\alpha}-1}$ (see proof of (5.18) in \cite{Ujms}).

Similar procedure of forming differences can be done by "moving to the right side"\, and
assuming
\begin{equation*} %\label{RP}=
R^*(\tau+r):=\sum_{j=0}^{\boldsymbol{\alpha}}(-1)^j
\binom{\boldsymbol{\alpha}}{j}P(\tau+r+j).
\end{equation*} Then
%\begin{equation}\label{RP1}
%P(\tau+r)=R^*(\tau+r)-\sum_{j=1}^{\boldsymbol{\alpha}}(-1)^j%\binom{\boldsymbol{\alpha}}{j}P(\tau+r+j)
%\end{equation} и, соответственно,
\begin{equation*}\label{RP1*}
P(\tau)=\sum_{r\ge 0} A_rR^*(\tau+r),\qquad \textrm{where}\quad R(\tau+r)^*=B_{m^*}^{(\boldsymbol{\alpha})}
(\cdot-\tau-r).\end{equation*}

Since we have $f\in L^1(-\infty,x)$ for any $x\in\mathbb{R}$,
also $${\rm supp}\,B_{m^*}^{(\boldsymbol{\alpha})}(\cdot-\tau+r+\boldsymbol{\alpha})=
{\rm supp}\,B_{m^*,\tau-r-\boldsymbol{\alpha}}^{(\boldsymbol{\alpha})}(\cdot)=
[\tau-r-\boldsymbol{\alpha},\tau-r+n^*+1]$$ and $B_{m^*,\tau-r-\boldsymbol{\alpha}}^{(k)}=0$ 
at the points
$\tau-r-\boldsymbol{\alpha}$ and $\tau-r+n^*+1$ for all $k=0,\ldots,\boldsymbol{\alpha}$
(see \S\,\ref{dopol}), then
\begin{equation}\label{18-12-1}
\int_{\mathbb{R}}
I_{+}^{\boldsymbol{\alpha}}f(x)\,
B_{m^*}^{(\boldsymbol{\alpha})}(x-\tau-r)\,dx=(-1)^{\boldsymbol{\alpha}}\int_{\mathbb{R}}f(y)
\,B_{m^*}(y-\tau+r+\boldsymbol{\alpha})\,dy.\end{equation}
Hence and with respect to the equality
$R(\tau-r)=(-1)^{\boldsymbol{\alpha}}B_{m^*}^{(\boldsymbol{\alpha})}
(\cdot-\tau+r+\boldsymbol{\alpha})$ (see \eqref{diff}), the first term in
\eqref{rysha+U} can be estimated as follows:
\begin{align*}%\label{06_07_02}
\Bigl\|\sum_{\tau\in\mathbb{Z}}
|\widetilde{\boldsymbol{\lambda}}_{0\tau}|\chi_{Q_{0\tau}}\Bigr\|_{L^p(\mathbb{R},u)}^p=&
\sum_{\tau\in\mathbb{Z}}u(Q_{0\tau})\bigl|\langle I_{+}^{\boldsymbol{\alpha}}f,
\widetilde{\mathbf{\Phi}}_{\tau}\rangle\bigr|^p\\
=&\sum_{\tau\in\mathbb{Z}}u(Q_{0\tau})\biggl[\Bigl|\sum_{r\ge 0}
A_r\int_{\mathbb{R}}
I_{+}^{\boldsymbol{\alpha}}f(x)\,
B_{m^*}^{(\boldsymbol{\alpha})}(x-\tau+r+\boldsymbol{\alpha})\,dx\Bigr|\biggr]^p\\
=&\Gamma(\boldsymbol{\alpha})\sum_{\tau\in\mathbb{Z}}u(Q_{0\tau})\biggl[\Bigl|\sum_{r\ge 0}
A_r\int_{\mathbb{R}}{f(y)}B_{m^*}
(y-\tau+r+\boldsymbol{\alpha})\,dy\Bigr|\biggr]^p\\
\lesssim& \sum_{\tau\in\mathbb{Z}}u(Q_{0\tau})\biggl[\sum_{r\ge 0} 
(r+1)^{\boldsymbol{\alpha}-1}\Bigl|\int_{\mathbb{R}}{f(y)}B_{m^*}
(y-\tau+r+\boldsymbol{\alpha})\,dy\Bigr|
\biggr]^p\\
=&\sum_{\tau\in\mathbb{Z}}u(Q_{0\tau})\biggl[\sum_{l\le\tau} 
(\tau-l+1)^{\boldsymbol{\alpha}-1}\Bigl|\int_{\mathbb{R}}{f(y)}B_{m^*}
(y+\boldsymbol{\alpha}-l)\,dy\Bigr|\biggr]^p
.\end{align*} Further, by virtue of \eqref{estN1}, 
\begin{equation}\label{06_07_03}\sum_{\tau\in\mathbb{Z}}u(Q_{0\tau})
\bigl|\langle I_{+}^{\boldsymbol{\alpha}}f,\widetilde{\mathbf{\Phi}}_{\tau}\rangle\bigr|^p
\lesssim [\mathbf{M}_{+}^{\boldsymbol{\alpha}}(1)+\mathbf{M}_{+}^{\boldsymbol{\alpha}}(0)]^p
\sum_{\tau\in\mathbb{Z}}v(Q_{0\tau})\biggl|\int_{\mathbb{R}}{f(y)}
B_{m^*}(y+\boldsymbol{\alpha}-\tau)\,dy\biggr|^p
\end{equation}  (see \cite[Theorem 5.1]{Ujms} or \cite[Theorems 1.6 \& 1.11]{PSU}). 
Using \eqref{dpI'} and relying on Lemma \ref{basa}, we derive an estimate of the form
\begin{equation*}%\label{30_11}
v(Q_{0\tau})=\int_{Q_{0\tau}}v\le\int_{\cup_{i=\tau-\boldsymbol{\alpha}}^\tau Q_{0i}}v\le 
\boldsymbol{\alpha}^{r_v}\int_{Q_{0(\tau-\boldsymbol{\alpha})}}v.\end{equation*} 
Then, by substituting $r=\tau-\boldsymbol{\alpha}$, we can transform \eqref{06_07_03} to
\begin{align*}%\label{06_07_03'} 
\Bigl\|\sum_{\tau\in\mathbb{Z}} |\lambda_{0\tau}|\chi_{Q_{0\tau}}\Bigr\|_{L^p(\mathbb{R},u)}&=
\biggl(\sum_{\tau\in\mathbb{Z}}u(Q_{0\tau})\bigl|\langle I_{+}^{\boldsymbol{\alpha}}f,
\widetilde{\mathbf{\Phi}}_{\tau}\rangle\bigr|^p
\biggr)^{1/p}\\&
\lesssim \mathfrak{C}_{+}^{\boldsymbol{\alpha}}(\kappa^*)\biggl(\sum_{r\in\mathbb{Z}}v(Q_{0r})
\biggl|\int_{\mathbb{R}}{f(y)}B_{m^*}(y-r)\,dy\biggr|^p\biggr)^{1/p}
%=\Biggl\|\sum_{\tau\in\mathbb{Z}} \biggl|\int_{\mathbb{R}}{f(y)}B_n(y-r)\,dy\biggr|\chi_{Q_{0\tau}}\Biggr\|_{L^p(\mathbb{R},u)}
.\end{align*}
%В \eqref{06_07_03} о
We set $\int_{\mathbb{R}}{f(y)}B_{m^*}(y-\tau)\,dy=:\langle f,
{\bar{\mathbf{\Phi}}} _{\tau}\rangle$ and name ${\bar{\mathbf{\Phi}}}$ as the first component
of a new wavelet system forming spline basis of order $m^*$. 
The second component of this system, as well as the elements it generates,
will be denoted by ${\bar{\mathbf{\Psi}}}^{\langle 1,0\rangle}$ and
${\bar{\mathbf{\Psi}}}_{( d-1)\tau}^{\langle 1,0\rangle}$, respectively. \label{est1sys}

The next step is to estimate $\widetilde{\boldsymbol{\lambda}}_{d\tau}^{\langle 1,1\rangle}$
for $d>0$ on the right--hand side
\eqref{rysha+U} via $\langle f,{\bar{\mathbf{\Psi}}}_{(d-1)\tau}^{\langle 1,0\rangle}\rangle$.
To do this, we write out an exact expression for the element
$\widetilde{\mathbf{\Psi}}^{\langle1,1\rangle}_{ (d-1)\tau}$,
based on \eqref{ksi*} and \eqref{formulll}:
\begin{align*}\label{formul}2^{-d/2}\gamma_{n^*}\mathbf{\Lambda}_{n^*}
\widetilde{\mathbf{\Psi}}^{\langle1,1\rangle}_{(d-1)\tau}(\cdot)=&
\gamma_{n^*}{\Psi}_{n^*,0,0;m^*(1),\boldsymbol{\alpha}(1)}(2^{d-1}\cdot-\tau)
\nonumber\\=&\sum_{|l|\le m^*}\lambda_{|l|}(m^*)(-1)^{|l|}
\sum_{|s|\le n^*}\lambda_{|s|}(n^*)(-1)^{|s|}\nonumber\\&\times
\sum_{k= 0}^{2\boldsymbol{\alpha}+n^*+1}(-1)^k\binom{1+m^*+2\boldsymbol{\alpha}}{k}\ 
B_{n^*}(2^{d}\cdot-2\tau +(n^*-k)-s-l).\end{align*} 
We will also need the Zhu--Vandermonde identity \cite[p. 15]{SKM}
\begin{equation}\label{ChVan}
\binom{r+s}{k}=\sum_{n=0}^k\binom{r}{n}\binom{s}{k-n},\qquad \textrm{where}
\quad r,s\in\mathbb{R}\quad \textrm{and}\quad k\in\mathbb{N},\end{equation}
from which it follows for $r=1+m^*$ and $s=\boldsymbol{\alpha}$ that
\begin{multline*}
\sum_{k= 0}^{2\boldsymbol{\alpha}+n^*+1}(-1)^k\binom{1+n^*+2\boldsymbol{\alpha}}{k} 
B_{n^*}(2^{d}x-2\tau +(m^*-k)-s-l)\\=
\sum_{k= 0}^{2\boldsymbol{\alpha}+n^*+1}(-1)^k\sum_{r=0}^k
\binom{1+m^*}{r}\binom{\boldsymbol{\alpha}}{k-r} 
B_{n^*}(2^{d}x-2\tau +(m^*-k)-s-l)\\=
\sum_{r= 0}^{m^*+1}(-1)^r\binom{1+m^*}{r}\sum_{m=0}^{\boldsymbol{\alpha}}(-1)^m
\binom{\boldsymbol{\alpha}}{m} 
B_{n^*}(2^{d}x-2\tau +(m^*-r-m)-s-l).
\end{multline*} Observe that
\begin{multline}\label{dopka24}
\sum_{m=0}^{\boldsymbol{\alpha}}(-1)^m\binom{\boldsymbol{\alpha}}{m} 
B_{n^*}(2^{d}x-2\tau +(m^*-r-m)-s-l)\\=
2^{-d\boldsymbol{\alpha}}B_{m^*}^{(\boldsymbol{\alpha})}(2^{d}x-2\tau +(m^*-r)-s-l).
\end{multline}
Also, 
\begin{align*}{\rm supp}\,B_{m^*}^{(\boldsymbol{\alpha})}(2^{d}x-2\tau +(m^*-r)-s-l)&=:
{\rm supp}\,B_{m^*,2\tau+s+l+r-m^*}^{(\boldsymbol{\alpha})}(2^{d}\cdot)
\\&=\left[\frac{2\tau+s+l+r-m^*}{2^d}, 
\frac{2\tau+s+l+r+1}{2^d}\right]\end{align*} and
$B_{m^*,2\tau+s+l+r-m^*}^{(\boldsymbol{\alpha})}(2^{d}\cdot)=0$
at the points $2^{-d}(2\tau+s+l+r-m^*)$ and $2^{-d}(2\tau+s+l+r+1)$ 
for all $k=0,\ldots,\boldsymbol{\alpha}$.
Therefore, and in view of integrability of $f$ on semiaxes $(-\infty,x)$ for any $x$,
\begin{multline*}%\label{09_07_01}
\int_{\mathbb{R}}\biggl( \int_{-\infty}^{x}\frac{{f(y)}\,dy}{(x-y)^{{1-\boldsymbol{\alpha}}}}
\biggr)\sum_{m=0}^{\boldsymbol{\alpha}}(-1)^m
\binom{\boldsymbol{\alpha}}{m} 
B_{n^*}(2^{d}x-2\tau +(m^*-r-m)-s-l)\\=
\int_{\frac{2\tau+s+l+r+m-m^*}{2^d}}
^{\frac{2\tau+s+l+r+m+1}{2^d}}
\biggl( \int_{-\infty}^{x}\frac{{f(y)}\,dy}{(x-y)^{{1-\boldsymbol{\alpha}}}}
\biggr)B_{m^*}^{(\boldsymbol{\alpha})}(2^{d}x-2\tau +(m^*-r)-s-l)\\
 =\int_{-\infty}^{\frac{2\tau+s+l+r+m-m^*}{2^d}}f(y)
\biggl(\int_{\frac{2\tau+s+l+r+m-m^*}{2^d}}
^{\frac{2\tau+s+l+r+m+1}{2^d}}
\frac{B_{m^*}^{(\boldsymbol{\alpha})}(2^{d}x-2\tau +(m^*-r)-s-l)\,dx}
{(x-y)^{{1-\boldsymbol{\alpha}}}}\biggr)dy\\
+\int_{\frac{2\tau+s+l+r+m-m^*}{2^d}}
^{\frac{2\tau+s+l+r+m+1}{2^d}}f(y)
\biggl(\int_{\frac{2\tau+s+l+r+m-m^*}{2^d}}
^{y}\frac{B_{m^*}^{(\boldsymbol{\alpha})}(2^{d}x-2\tau +(m^*-r)-s-l)\,dx}
{(x-y)^{{1-\boldsymbol{\alpha}}}}\biggr)dy\\
=\frac{\Gamma(\boldsymbol{\alpha})}{2^{d\boldsymbol{\alpha}}}
\int_{\frac{2\tau+s+l+r-m^*}{2^d}}
^{\frac{2\tau+s+l+r+m+1}{2^d}}f(y)
B_{m^*}(2^{d}y-2\tau +(m^*-r)-s-l).\end{multline*} 
From this we obtain that
\begin{multline*}\label{kra1}\frac{{{2^{d\boldsymbol{\alpha}}}}\mathbf{\Lambda}_{n^*}}
{2^{d/2}\Gamma(\boldsymbol{\alpha})}\langle I_{+}^{\boldsymbol{\alpha}}f,
\widetilde{\mathbf{\Psi}}^{\langle1,1\rangle}_{(d-1)\tau}\rangle=
\frac{{{2^{d\boldsymbol{\alpha}}}\gamma_{n^*}}}{\Gamma(\boldsymbol{\alpha})
}%{\Gamma(\boldsymbol{\alpha})}
\langle I_{+}^{\boldsymbol{\alpha}}f,{\Psi}_{n^*,0,0;m^*(1),\boldsymbol{\alpha}(1)}(2^{d-1}
\cdot-\tau)\rangle\\=
\sum_{|l|\le m^*}\lambda_{|l|}(m^*)(-1)^{|l|}
\sum_{|s|\le n^*}\lambda_{|s|}(n^*)(-1)^{|s|}
\sum_{r= 0}^{m^*+1}(-1)^r\binom{1+m^*}{r}\\\times
\int_{\mathbb{R}}{f(y)}B_{m^*}(2^{d}y-2\tau +(m^*-r)-s-l)\,dy=
\frac{\gamma_{m^*}\mathbf{\Lambda}_{m^*}}{2^{d/2}}
\langle f,
{\bar{\mathbf{\Psi}}}^{\langle1,0\rangle}_{(d-1)\tau}\rangle,\end{multline*}
where ${\bar{\mathbf{\Psi}}}^{\langle 1,0\rangle}$ is specified by 
$\Psi_{m^*,0,0;n^*(1),\boldsymbol{\alpha}(0)}$, and we can write that
\begin{equation*}%\label{12_07_01}
2^{pd\boldsymbol{\alpha}}\sum_{\tau\in\mathbb{Z}}u(Q_{(d-1)\tau})\bigl|
\langle I_{+}^{\boldsymbol{\alpha}}f,\widetilde{\mathbf{\Psi}}_{(d-1)\tau}^{\langle 
1,1\rangle}\rangle\bigr|^p
=\sum_{\tau\in\mathbb{Z}}u(Q_{(d-1)\tau})
\biggl|\int_{\mathbb{R}}{f(y)}
{\bar{\mathbf{\Psi}}}^{\langle1,0\rangle}_{(d-1)\tau}(y)\,dy
\biggr|^p.\end{equation*} 
Applying the condition \eqref{estN1}, we obtain 
\begin{multline*}\label{Proof2'}2^{-d/2}2^{d(-\kappa^*+\boldsymbol{\alpha})}\Bigl\|\sum_{\tau\in\mathbb{Z}} |
\widetilde{\boldsymbol{\lambda}}_{d\tau}^{\langle 1,1\rangle}|\chi_{Q_{(d-1)\tau}}
\Bigr\|_{L^p(\mathbb{R},u)}=
2^{d(-\kappa^*+\boldsymbol{\alpha})}\biggl(\sum_{\tau\in\mathbb{Z}}u(Q_{(d-1)\tau})
\bigl|\langle I_{+}^{\boldsymbol{\alpha}}f,\widetilde{\mathbf{\Psi}}_{(d-1)\tau}^{\langle
1,1\rangle}\rangle\bigr|^p
\biggr)^{1/p}\\\simeq 2^{-d\kappa^*}
\biggl(\sum_{\tau\in\mathbb{Z}}u(Q_{(d-1)\tau})
\bigl|\langle f,
{\bar{\mathbf{\Psi}}}^{\langle1,0\rangle}_{(d-1)\tau}\rangle\bigr|
^p \biggr)^{1/p}\\
\lesssim \sup_{d\in\mathbb{N}}\mathfrak{M}(d,\kappa^*;v,u)
\biggl(\sum_{\tau\in\mathbb{Z}}v(Q_{(d-1)\tau})\bigl|\langle f,{\bar{\mathbf{\Psi}}}^{\langle
1,0\rangle}_{(d-1)\tau}\rangle\bigr|^p\biggr)^{1/p}\\
\lesssim \mathfrak{C}_+^{\boldsymbol{\alpha}}(\kappa^*)
\biggl(\sum_{\tau\in\mathbb{Z}}v(Q_{(d-1)\tau})\bigl|\langle f,{\bar{\mathbf{\Psi}}}^{\langle
1,0\rangle}_{(d-1)\tau}\rangle\bigr|^p\biggr)^{1/p}
=
\mathfrak{C}_+^{\boldsymbol{\alpha}}(\kappa^*)
\Bigl\|\sum_{\tau\in\mathbb{Z}} \chi_{Q_{(d-1)\tau}}
\bigl|\langle f,{\bar{\mathbf{\Psi}}}^{\langle1,0\rangle}_{(d-1)\tau}\rangle\bigr|
\Bigr\|_{L^p(\mathbb{R},v)}.
\end{multline*}
Thus, ${\bar{\mathbf{\Psi}}}^{\langle1,0\rangle}_{(d-1)\nu}$
together with ${\bar{\mathbf {\Phi}}}_{\nu}$ (see p. \pageref{est1sys})
forms a spline wavelet system of of order $m^*$.
The sequential norms corresponding to this system, by Theorem \ref{main}, are equivalent
to the norm of $f$ in the space
$B_{pq}^{s+\kappa^*-\boldsymbol{\alpha}, v}(\mathbb{R})$. This means that
assertion (i) is proven.

{\it Necessity.} Suppose that the operator $I_{+}^{\boldsymbol{\alpha}}$ is bounded from
${B}_{pq}^{s+\kappa^*-\boldsymbol{\alpha},v}(\mathbb{R})$ with some $\kappa^*\in\mathbb{R}$ to
${B}_{pq}^{s,u}(\mathbb{R})$. Accordingly,
the following inequality holds
\begin{equation}\label{4T1more+BEST}
\|I_{+}^{\boldsymbol{\alpha}}f\|_{{B}_{pq}^{s,u}(\mathbb{R})}\lesssim
\mathfrak{C}\|f\|_{{B}_{pq}^{s+\kappa^*-\boldsymbol{\alpha},v}(\mathbb{R})},
\end{equation} where $\mathfrak{C}>0$ is the best constant in \eqref{4T1more+BEST}, i.e.
$\mathfrak{C}=\|I_{\pm}^{\boldsymbol{\alpha}}\|_{{B}_{pq}^{s+\kappa^*-\boldsymbol{\alpha},v}
(\mathbb{R})\to
{B}_{pq}^{s,u}(\mathbb{R})}$. 
We will show that
\begin{equation}\label{snizu}
\mathfrak{C}\gtrsim \mathfrak{C}_+(\kappa^*)\ge \max\{\mathbf{M}_+^{\boldsymbol{\alpha}}(1),
\mathbf{M}_+^{\boldsymbol{\alpha}}(0),\sup_{d\in\mathbb{N}}\mathfrak{M}(d,\kappa^*;v,u)\}.
\end{equation}   

Let $n^\ast\in\mathbb{N}$ and $m^\ast\in\mathbb{N}$ satisfy the condition \eqref{condB}
with respect to the parameters of the spaces $B^{s,u}_{pq}(\mathbb{R})$ and
$B^{s+\kappa^*-\boldsymbol{\alpha},v}_{pq}( \mathbb{R})$, respectively.
By theorem \ref{main'},
\begin{equation}\label{24_05}
\|I_{+}^{\boldsymbol{\alpha}}f\|_{B^{s,u}_{pq}(\mathbb{R})}\approx
\|\widetilde{\boldsymbol{\lambda}}^{\langle \widetilde{\boldsymbol{a}}\rangle}_{(I_{+}^{\boldsymbol{\alpha}}f)}
\|_{b_{pq}^{s}(u)},\end{equation}
where the sequence
$\widetilde{\boldsymbol{\lambda}}^{\langle \widetilde{\boldsymbol{a}}\rangle}_{(I_{+}^
{\boldsymbol{\alpha}}f)}$
is defined by the elements
\begin{equation}\label{asty'++}
\widetilde{\boldsymbol{\lambda}}_{0\tau}^{\langle \widetilde{\boldsymbol{a}}\rangle}=
\langle I_{+}^{\boldsymbol{\alpha}}f,
\widetilde{\mathbf{\Phi}}_{\tau}^{\langle \widetilde{\boldsymbol{a}}\rangle}\rangle\quad(\tau\in\mathbb{Z}),\qquad   
\widetilde{\boldsymbol{\lambda}}_{d\tau}^{\langle\widetilde{\boldsymbol{a}}\rangle}=2^{d/2}\langle 
I_{+}^{\boldsymbol{\alpha}}f,\widetilde{\mathbf{\Psi}}^{\langle\widetilde{\boldsymbol{a}}\rangle}_{(d-1)\tau}
\rangle\quad(d\in\mathbb{N},\,\tau\in\mathbb{Z}).\end{equation}
Here, $\widetilde{\boldsymbol{\Phi}}^{\langle\widetilde{\boldsymbol{a}}\rangle}$
are defined by $B_{n*,\widetilde{\boldsymbol{a}}}$
and $\widetilde{\boldsymbol{\Psi}}^{\langle \widetilde{\boldsymbol{a}}\rangle}$
(see \eqref{vazhnoo}) --- via
$\Psi_{n^\ast,\widetilde{\boldsymbol{a}},0}$, where
$\widetilde{\boldsymbol{a}}$ can take values $0$ or $\pm 1/2$.
Based on the above, we can write:
\begin{equation}\label{rysha+}\|I_{+}^{\boldsymbol{\alpha}}f\|_{B^{s,u}_{pq}(\mathbb{R})}
\approx 
\Bigl\|\sum_{\tau\in\mathbb{Z}}
|\widetilde{\boldsymbol{\lambda}}^{\langle \widetilde{\boldsymbol{a}}\rangle}_{0\tau}|\chi_{Q_{0\tau}}\Bigr\|_{L^p(\mathbb{R},u)}+
\biggl(\sum_{d\in\mathbb{N}}2^{qd s}\Bigl\|\sum_{\tau\in\mathbb{Z}}
|\widetilde{\boldsymbol{\lambda}}_{d\tau}^{\langle\widetilde{\boldsymbol{a}}\rangle}|\chi_{Q_{(d-1)\tau}}
\Bigr\|_{L^p(\mathbb{R},u)}^q\biggr)^{1/q}.\end{equation}
Moreover, also based on Theorem \ref{main'},
$$
\|f\|_{{B}_{pq}^{s+\kappa^*-\boldsymbol{\alpha},v}(\mathbb{R})}\approx
\|{\boldsymbol{\lambda}}^{\langle \boldsymbol{a}\rangle}_f\|_{b_{pq}^{s+\kappa^*-\boldsymbol{\alpha}}(v)},$$
where the sequence ${\boldsymbol{\lambda}}^{\langle \boldsymbol{a}\rangle}_f$ with 
$\boldsymbol{a}
\in\{0,\pm1/2\}$ is defined by elements 
\begin{equation}\label{asty'25}
{\boldsymbol{\lambda}}_{0\tau}^{\langle \boldsymbol{a}\rangle}=\langle f,
{\mathbf{\Phi}}_{\tau}^{\langle \boldsymbol{a}\rangle}
\rangle\quad(\tau\in\mathbb{Z}),\qquad   
{\boldsymbol{\lambda}}_{d\tau}^{\langle \boldsymbol{a}\rangle}=2^{d/2}\langle 
f,{\mathbf{\Psi}}^{\langle \boldsymbol{a}\rangle}_{(d-1)\tau}
\rangle\quad(d\in\mathbb{N},\,\tau\in\mathbb{Z}).\end{equation}
In this case, ${\boldsymbol{\Phi}}^{\langle \boldsymbol{a}\rangle}$ is defined by
$B_{m*,\boldsymbol{a}}$,
and ${\boldsymbol{\Psi}}^{\langle \boldsymbol{a}\rangle}$ (see \eqref{vazhnooW}) --- via
$\Psi_{m^\ast,\boldsymbol{a},0}$.
Thus, for the right--hand side of \eqref{4T1more+BEST} we can write that
\begin{equation}\label{rysha+R}\|f\|_{{B}_{pq}^{s+\kappa^*-\boldsymbol{\alpha},v}(\mathbb{R})}\approx 
\Bigl\|\sum_{\tau\in\mathbb{Z}}
|{\boldsymbol{\lambda}}_{0\tau}^{\langle \boldsymbol{a}\rangle}|\chi_{Q_{0\tau}}\Bigr\|_{L^p(\mathbb{R},v)}+
\biggl(\sum_{d\in\mathbb{N}}2^{qd (s+\kappa^*-\boldsymbol{\alpha})}\Bigl\|\sum_{\tau\in\mathbb{Z}}
|{\boldsymbol{\lambda}}_{d\tau}^{\langle \boldsymbol{a}\rangle}|\chi_{Q_{(d-1)\tau}}
\Bigr\|_{L^p(\mathbb{R},v)}^q\biggr)^{1/q}.\end{equation}

To prove \eqref{snizu}
we first assume that the maximum on the right side of \eqref{snizu} is equal to 
$$\mathbf{M}_+^{\boldsymbol{\alpha}}(0)=
\sup_{\tau\in\mathbb{Z}}
\biggl(\sum_{r\ge\tau}
u(Q_{0r})\biggr)^{\frac{1}{p}}\biggl(\sum_{r\le\tau} (\tau-r+1)^{p'(\boldsymbol{\alpha}-1)}
[v(Q_{0r})]^{1-p'}\biggr)^{\frac{1}{p'}}.$$ Then fix $n^*$ and 
$m^*=n^\ast+\boldsymbol{\alpha}$, and leave only the first of the two terms 
on the right side of
\eqref{rysha+}. Due to the procedure of forming iterative differences (to the left) of order
$\boldsymbol{\alpha}$ with coefficients  $$A_r(\boldsymbol{\alpha})=
\sum_{j=1}^r(-1)^{j-1}\binom{\boldsymbol{\alpha}}{j}A_{r-j}(\boldsymbol{\alpha}),
\qquad\textrm{where}\quad
A_0(\boldsymbol{\alpha})=1,$$
(see e.g. \eqref{18-12-1} and \eqref{diff}) % или утверждения (i) теоремы 5.5. в \cite{Ujms}, 
%а также \S\,5 в \cite{Ujms}) 
the following equalities are true:
\begin{multline*}%\label{06_07_02}
\frac{m^*!}{n^*!}\Bigl\|\sum_{\tau\in\mathbb{Z}}
|\widetilde{\boldsymbol{\lambda}}_{0\tau}^{\langle\widetilde{\boldsymbol{a}}\rangle}|\chi_{Q_{0\tau}}\Bigr\|_{L^p(\mathbb{R},u)}^p=
\sum_{\tau\in\mathbb{Z}}u(Q_{0\tau})\bigl|\langle I_{+}^{\boldsymbol{\alpha}}f,
\widetilde{\mathbf{\Phi}}_{\tau}^{\langle\widetilde{\boldsymbol{a}}\rangle}\rangle\bigr|^p\\
= \sum_{\tau\in\mathbb{Z}}u(Q_{0\tau})\biggl|\sum_{r\ge 0} 
A_r(\boldsymbol{\alpha})\int_{\mathbb{R}}{f(y)}B_{m^*}
(y-\tau+\boldsymbol{\alpha}-\widetilde{\boldsymbol{a}}+r)\,dy
\biggr|^p\\= \sum_{\tau\in\mathbb{Z}}u(Q_{0\tau})\biggl|\sum_{r\ge 0} 
A_r(\boldsymbol{\alpha})\langle f, B_{m^*, \tau-\boldsymbol{\alpha}+\widetilde{\boldsymbol{a}}-r}\rangle
\biggr|^p= \sum_{\tau\in\mathbb{Z}}u(Q_{0\tau})\biggl|\sum_{l\le \tau} 
A_{\tau-l}(\boldsymbol{\alpha})\langle f, B_{m^*, l-\boldsymbol{\alpha}+\widetilde{\boldsymbol{a}}}\rangle
\biggr|^p.\end{multline*} Then we obtain from \eqref{4T1more+BEST}, \eqref{rysha+}
and \eqref{rysha+R}:
\begin{multline*}%\label{4T1more+BEST1}
\biggl(\sum_{\tau\in\mathbb{Z}}u(Q_{0\tau})\biggl|\sum_{l\le \tau} 
A_{\tau-l}(\boldsymbol{\alpha})\langle f, B_{m^*, l-\boldsymbol{\alpha}+\widetilde{\boldsymbol{a}}}\rangle
\biggr|^p\biggr)^{1/p}\simeq
\Bigl\|\sum_{\tau\in\mathbb{Z}}
|\widetilde{\boldsymbol{\lambda}}_{0\tau}^{\langle\widetilde{\boldsymbol{a}}\rangle}|\chi_{Q_{0\tau}}\Bigr\|_{L^p(\mathbb{R},u)}\\\lesssim
\|I_{+}^{\boldsymbol{\alpha}}f\|_{{B}_{pq}^{s,u}(\mathbb{R})}\lesssim 
\mathfrak{C}\|f\|_{{B}_{pq}^{s+\kappa^*-\boldsymbol{\alpha},v}(\mathbb{R})}
\\\approx\mathfrak{C}\Biggl[
\Bigl\|\sum_{\tau\in\mathbb{Z}}
|{\boldsymbol{\lambda}}_{0\tau}^{\langle \boldsymbol{a}\rangle}|\chi_{Q_{0\tau}}\Bigr\|_{L^p(\mathbb{R},v)}+
\biggl(\sum_{d\in\mathbb{N}}2^{qd (s+\kappa^*-\boldsymbol{\alpha})}\Bigl\|\sum_{\tau\in\mathbb{Z}}
|{\boldsymbol{\lambda}}_{d\tau}^{\langle \boldsymbol{a}\rangle}|\chi_{Q_{(d-1)\tau}}
\Bigr\|_{L^p(\mathbb{R},v)}^q\biggr)^{1/q}\Biggr]\\=\mathfrak{C}\Biggl[
\biggl(\sum_{\tau\in\mathbb{Z}}v(Q_{0\tau})
\bigl|\langle f,
B_{m^*, \tau+\boldsymbol{a}}\rangle\bigr|^p\biggr)^{1/p}+
\biggl(\sum_{d\in\mathbb{N}}2^{qd (s+\kappa^*-\boldsymbol{\alpha})}
\biggl(\sum_{\tau\in\mathbb{Z}}v(Q_{(d-1)\tau})
|\langle f, \boldsymbol{\Psi}_{(d-1)\tau}^{\langle \boldsymbol{a}\rangle}\rangle|
\biggr)^{q/p}\biggr)^{1/q}\Biggr].
\end{multline*} On the strength of Lemma \ref{basa} and property \eqref{dpI'},
$$\sum_{\tau\in\mathbb{Z}}v(Q_{0\tau})
\bigl|\langle f,
B_{m^*, \tau+\boldsymbol{a}}\rangle\bigr|^p\lesssim 
\sum_{l\in\mathbb{Z}}v(Q_{0l})
\bigl|\langle f,
B_{m^*,l-\boldsymbol{\alpha}+\boldsymbol{a}}\rangle\bigr|^p.$$
Thus,
\begin{multline}\label{4T1more+BEST1}
\biggl(\sum_{\tau\in\mathbb{Z}}u(Q_{0\tau})\biggl|\sum_{l\le \tau} 
A_{\tau-l}(\boldsymbol{\alpha})\langle f, B_{m^*, l-\boldsymbol{\alpha}+
\widetilde{\boldsymbol{a}}}\rangle
\biggr|^p\biggr)^{1/p}\lesssim
\mathfrak{C} 
\Biggl[\biggl(\sum_{l\in\mathbb{Z}}v(Q_{0l})
\bigl|\langle f,
B_{m^*,l-\boldsymbol{\alpha}+\boldsymbol{a}}\rangle\bigr|^p\biggr)^{1/p}\\+
\biggl(\sum_{d\in\mathbb{N}}2^{qd (s+\kappa^*-\boldsymbol{\alpha})}
\biggl(\sum_{\tau\in\mathbb{Z}}v(Q_{(d-1)\tau})
|\langle f, \boldsymbol{\Psi}_{(d-1)\tau}^{\langle \boldsymbol{a}\rangle}\rangle|
\biggr)^{q/p}\biggr)^{1/q}\Biggr].
\end{multline} 

For $A_r(m)$ with fixed
$m\in\mathbb{N}$, the following lower bound holds for all $r\in\mathbb{N}$:
\begin{equation}\label{AR}
(m-1)!A_r(m)\ge (1+r)^{m-1}=\sum_{s=0}^{m-1}\frac{(m-1)!\,r^s}{s!(m-1-s)!}.
\end{equation} We will show that this is so, taking into account (see \cite[(5.20)]{Ujms}) that
$A_1(m)=m$ and 
\begin{equation}\label{AR1}
A_r(m)=\sum_{1\le l\le m}A_{r-1}(l),\qquad(m\in\mathbb{N}).\end{equation}
For $r=1$, \eqref{AR} holds by \eqref{AR1} for all $m=m_0\in\mathbb{N}$. Suppose
that \eqref{AR} holds for $r=k-1$ with some $k>1$ and any $m=l$, where $1\le l\le m_0$.
Then it follows from \eqref{AR1} that
$$A_k(m_0)=\sum_{1\le l\le m_0}A_{k-1}(l)\ge
\sum_{1\le l\le m_0}\frac{k^{l-1}}{(l-1)!}=\sum_{s=0}^{m_0-1}\frac{k^{s}}{s!}
\ge \sum_{s=0}^{m_0-1}\frac{k^{s}}{s!(m_0-1-s)!}=\frac{(1+k)^{m_0-1}}{(m_0-1)!}.$$

Next, we set $\boldsymbol{a}=\widetilde{\boldsymbol{a}}=0$ and denote by
$\mathfrak{B}_{m^*}$ the class of functions from 
${B}_{pq}^{s+\kappa^*-\boldsymbol{\alpha},v}(\mathbb{R})$ which can be represented as 
a finite linear combination of integer shifts of $B-$splines
of order $m^*$. Due to the semi--orthogonality of the basis elements on the right--hand 
side of \eqref{rysha+R} (see \eqref{ForRepr'W}),
the restriction of \eqref{4T1more+BEST1} to the sequences
$\boldsymbol{\lambda}_f^{\langle 0\rangle}$ with $f$ from the class $\mathfrak{B}_{m^*}$ 
has the form \begin{equation}\label{dopka11-0}
\biggl(\sum_{\tau\in\mathbb{Z}}u(Q_{0\tau})\biggl|\sum_{l\le \tau} 
A_{\tau-l}(\boldsymbol{\alpha})\langle f, B_{m^*, l-\boldsymbol{\alpha}}\rangle
\biggr|^p\biggr)^{1/p}\lesssim
\mathfrak{C}_{\mathfrak{B}_{m^*}}
\biggl(\sum_{l\in\mathbb{Z}}v(Q_{0l})
\bigl|\langle f,
B_{m^*,l-\boldsymbol{\alpha}}\rangle\bigr|^p\biggr)^{1/p}. %,\quad f\in\mathfrak{B}_{m^*}.
\end{equation} Besides, the best constant $\mathfrak{C}_{\mathfrak{B}_{m^*}}$ of this 
inequality does not exceed $\mathfrak{C}$.

For a sufficiently large $R\in\mathbb{N}$, we substitute into \eqref{dopka11-0} a function $f^*_R$
of the form
\begin{multline}\label{function*}
f^*_R(y)=\sum_{-R\le\tau\le\nu}(\nu-\tau+1)^{(\boldsymbol{\alpha}-1)(p'-1)} [v(Q_{0\tau})]^{1-p'}
B_{m^*,\tau-\boldsymbol{\alpha}}(y)\\=
\sum_{\tau\le\nu}(\nu-\tau+1)^{(\boldsymbol{\alpha}-1)(p'-1)} [v(Q_{0\tau})]^{1-p'}
B_{m^*}(y-\tau+\boldsymbol{\alpha})\qquad (\nu\in\mathbb{Z}).
\end{multline} This will lead us to inequality 
\begin{equation}\label{dopka11}
\biggl(\sum_{\tau\in\mathbb{Z}}u(Q_{0\tau})\biggl|\sum_{l\le \tau} 
A_{\tau-l}(\boldsymbol{\alpha})\langle f^*_R, B_{m^*, l-\boldsymbol{\alpha}}\rangle
\biggr|^p\biggr)^{1/p}\lesssim
\mathfrak{C}_{\mathfrak{B}_{m^*}} 
\biggl(\sum_{l\in\mathbb{Z}}v(Q_{0l})
\bigl|\langle f^*_R,
B_{m^*,l-\boldsymbol{\alpha}}\rangle\bigr|^p\biggr)^{1/p}.
\end{equation} Due to the non--negativity of $f^*_R$ and $B_{m^*}$ and on the basis of \eqref{AR}
\begin{equation*}\sum_{l\le \tau} 
A_{\tau-l}(\boldsymbol{\alpha})\langle f^*_R, B_{m^*, l-\boldsymbol{\alpha}}\rangle
\gtrsim
\sum_{l\le \tau} 
(\tau-l+1)^{\boldsymbol{\alpha}-1}\langle f^*_R, B_{m^*, l-\boldsymbol{\alpha}}\rangle.
\end{equation*}
Besides, 
\begin{multline*}\biggl(\sum_{\tau\in\mathbb{Z}}u(Q_{0\tau})\biggl|\sum_{l\le \tau} 
(\tau-l+1)^{\boldsymbol{\alpha}-1}\langle f^*_R, B_{m^*, l-\boldsymbol{\alpha}}\rangle
\biggr|^p\biggr)^{1/p}\\
\ge \Bigl(\sum_{\tau\ge \nu
}u(Q_{0\tau})\Bigr)^{1/p}
\sum_{l\le \nu} 
(\nu-l+1)^{\boldsymbol{\alpha}-1}\langle f^*_R, 
B_{m^*, l-\boldsymbol{\alpha}}\rangle.
\end{multline*} Thus, \eqref{dopka11} is transformable into
\begin{equation}\label{4T1more+BEST2}\Bigl(\sum_{\tau\ge \nu
}u(Q_{0\tau})\Bigr)^{1/p}
\sum_{l\le \nu} 
(\nu-l+1)^{\boldsymbol{\alpha}-1}\langle f^*_R, 
B_{m^*, l-\boldsymbol{\alpha}}\rangle
\lesssim
\mathfrak{C}_{\mathfrak{B}_{m^*}}
\biggl(\sum_{l\in\mathbb{Z}}v(Q_{0l})
\bigl|\langle f^*_R,
B_{m^*,l-\boldsymbol{\alpha}}\rangle\bigr|^p\biggr)^{1/p}.
\end{equation}
Further, we write for the sum depending on $f^*_R$ on the left--hand side 
\eqref{4T1more+BEST2}:
\begin{multline*}
\sum_{l\le \nu} 
(\nu-l+1)^{\boldsymbol{\alpha}-1}\langle f^*_R, 
B_{m^*, l-\boldsymbol{\alpha}}\rangle\\
=\sum_{l\le \nu} 
(\nu-l+1)^{\boldsymbol{\alpha}-1}\sum_{-R\le\tau\le\nu}(\nu-\tau+1)^{(\boldsymbol{\alpha}-1)(p'-1)}
[v(Q_{0\tau})]^{1-p'}\langle 
B_{m^*,\tau-\boldsymbol{\alpha}}, 
B_{m^*, l-\boldsymbol{\alpha}}\rangle\\
=\sum_{-R\le \tau\le\nu}(\nu-\tau+1)^{(\boldsymbol{\alpha}-1)(p'-1)} [v(Q_{0\tau})]^{1-p'}
\sum_{l\le \nu} 
(\nu-l+1)^{\boldsymbol{\alpha}-1}\langle 
B_{m^*,\tau-\boldsymbol{\alpha}}, 
B_{m^*, l-\boldsymbol{\alpha}}\rangle\\
\ge\sum_{-R\le \tau\le\nu}(\nu-\tau+1)^{(\boldsymbol{\alpha}-1)(p'-1)} [v(Q_{0\tau})]^{1-p'}
\sum_{l\le \tau} 
(\nu-l+1)^{\boldsymbol{\alpha}-1}\langle 
B_{m^*,\tau-\boldsymbol{\alpha}}, 
B_{m^*, l-\boldsymbol{\alpha}}\rangle\\
\ge\sum_{-R\le \tau\le\nu}(\nu-\tau+1)^{(\boldsymbol{\alpha}-1)p'} [v(Q_{0\tau})]^{1-p'}\sum_{l\le \tau} 
\langle 
B_{m^*,\tau-\boldsymbol{\alpha}}, 
B_{m^*, l-\boldsymbol{\alpha}}\rangle.
\end{multline*} Due to \eqref{phiphi} and orthonormality of
$\{\phi_{m^*,\tau}\}_{\tau\in\mathbb{Z}}$,
\begin{multline}\label{sm}
\beta_{m^*}^2\sum_{l\le \tau} 
\langle 
B_{m^*,\tau-\boldsymbol{\alpha}}, 
B_{m^*, l-\boldsymbol{\alpha}}\rangle\\=
\sum_{\varkappa'=0}^{m^*}\boldsymbol{\alpha}'_{\varkappa'}
\sum_{\varkappa''=0}^{m^*}\boldsymbol{\alpha}'_{\varkappa''}
\sum_{l\le \tau}\int
\phi_{m^*}(y+{\varkappa'}-\tau+\boldsymbol{\alpha})
\phi_{m^*}(y+\varkappa''-l+\boldsymbol{\alpha})\\=
\sum_{\varkappa=0}^{m^*}(\boldsymbol{\alpha}'_{\varkappa})^2
+\sum_{\varkappa''-\varkappa'=1}\boldsymbol{\alpha}'_{\varkappa'}
\boldsymbol{\alpha}'_{\varkappa''}+\ldots
+\sum_{\varkappa''-\varkappa'=m^*}\boldsymbol{\alpha}'_{\varkappa'}
\boldsymbol{\alpha}'_{\varkappa''}\\\gtrsim
\sum_{\varkappa=0}^{m^*}(\boldsymbol{\alpha}'_{\varkappa})^2
+\sum_{|\varkappa'-\varkappa''|=1}\boldsymbol{\alpha}'_{\varkappa'}
\boldsymbol{\alpha}'_{\varkappa''}+\ldots
+\sum_{|\varkappa'-\varkappa''|=m^*}\boldsymbol{\alpha}'_{\varkappa'}
\boldsymbol{\alpha}'_{\varkappa''}=
\Bigl(\sum_{\varkappa=0}^{m^*}\boldsymbol{\alpha}'_{\varkappa}\Bigr)^2=\beta_{m^*}^2.
\end{multline}
Thus, the left--hand side of \eqref{4T1more+BEST2} 
is estimated with some constant from below by
\begin{equation}\label{dopka12}
\Bigl(\sum_{\tau\ge \nu
}u(Q_{0\tau})\Bigr)^{1/p}
\sum_{-R\le \tau\le\nu}(\nu-\tau+1)^{(\boldsymbol{\alpha}-1)p'} [v(Q_{0\tau})]^{1-p'}.
\end{equation}

To estimate the right--hand side of \eqref{4T1more+BEST2} from above, we write, using the 
H\'{o}lder inequality,
\begin{multline*}
\bigl|\langle f^*_R,
B_{m^*,l-\boldsymbol{\alpha}}\rangle\bigr|^p=
\Bigl[\sum_{-R\le \tau\le\nu}(\nu-\tau+1)^{(\boldsymbol{\alpha}-1)(p'-1)} [v(Q_{0\tau})]^{1-p'}
\langle B_{m^*}(\cdot-\tau+\boldsymbol{\alpha}),
B_{m^*}(\cdot-l+\boldsymbol{\alpha})\rangle\Bigr]^p\\\le
\sum_{-R\le \tau\le\nu}(\nu-\tau+1)^{(\boldsymbol{\alpha}-1)p'} [v(Q_{0\tau})]^{-p'}
\langle B_{m^*}(\cdot-\tau+\boldsymbol{\alpha}),
B_{m^*}(\cdot-l+\boldsymbol{\alpha})\rangle\\\times
\Bigl[\sum_{-R\le r\le\nu}\langle B_{m^*}(\cdot-r+\boldsymbol{\alpha}),
B_{m^*}(\cdot-l+\boldsymbol{\alpha})\rangle\Bigr]^{p-1}.
\end{multline*} Hence, due to bounded supports of $B_{m^*}$,
 \begin{multline*}
\sum_{l\in\mathbb{Z}}v(Q_{0l})
\bigl|\langle f^*_R,
B_{m^*,l-\boldsymbol{\alpha}}\rangle\bigr|^p\\\le
\sum_{l\in\mathbb{Z}}v(Q_{0l})
\sum_{-R\le \tau\le\nu}(\nu-\tau+1)^{(\boldsymbol{\alpha}-1)p'} [v(Q_{0\tau})]^{-p'}
\langle B_{m^*}(\cdot-\tau+\boldsymbol{\alpha}),
B_{m^*}(\cdot-l+\boldsymbol{\alpha})\rangle\\\times
\Bigl[\sum_{-R\le r\le\nu}\langle B_{m^*}(\cdot-r+\boldsymbol{\alpha}),
B_{m^*}(\cdot-l+\boldsymbol{\alpha})\rangle\Bigr]^{p-1}\\=
\sum_{-R\le \tau\le\nu}(\nu-\tau+1)^{(\boldsymbol{\alpha}-1)p'} [v(Q_{0\tau})]^{-p'}
\sum_{l\in\mathbb{Z}}v(Q_{0l})\langle B_{m^*}(\cdot-\tau+\boldsymbol{\alpha}),
B_{m^*}(\cdot-l+\boldsymbol{\alpha})\\\times
\Bigl[\sum_{-R\le r\le\nu}\langle B_{m^*}(\cdot-r+\boldsymbol{\alpha}),
B_{m^*}(\cdot-l+\boldsymbol{\alpha})\rangle\Bigr]^{p-1}\\\le
\sum_{-R\le \tau\le\nu}(\nu-\tau+1)^{(\boldsymbol{\alpha}-1)p'} [v(Q_{0\tau})]^{-p'}
\sum_{|l-\tau|\le m^*}v(Q_{0l})\langle B_{m^*}(\cdot-\tau+\boldsymbol{\alpha}),
B_{m^*}(\cdot-l+\boldsymbol{\alpha})\\\times
\Bigl[\sum_{r\le\nu}\langle B_{m^*}(\cdot-r+\boldsymbol{\alpha}),
B_{m^*}(\cdot-l+\boldsymbol{\alpha})\rangle\Bigr]^{p-1}.
\end{multline*} For a fixed $l$, due to \eqref{phiphi} and orthonormality of
$\{\phi_{m^*,\tau}\}_{\tau\in\mathbb{Z}}$ (see \eqref{sm}),
\begin{equation*}
\sum_{r\le\nu}\langle B_{m^*}(\cdot-r+\boldsymbol{\alpha}),
B_{m^*}(\cdot-l+\boldsymbol{\alpha})\rangle\le
\sum_{|r-l|\le m^*}\langle B_{m^*}(\cdot-r+\boldsymbol{\alpha}),
B_{m^*}(\cdot-l+\boldsymbol{\alpha})\rangle\le 1.
\end{equation*}
By virtue of Lemma \ref{basa} and property \eqref{dpI'} we have
$v(Q_{0l})\lesssim
v(Q_{0\tau})$ for any $|l-\tau|\le m^*$. Therefore, for the right--hand side
\eqref{4T1more+BEST2} it holds that
\begin{multline*}
\sum_{l\in\mathbb{Z}}v(Q_{0l})
\bigl|\langle f^*_R,
B_{m^*,l-\boldsymbol{\alpha}}\rangle\bigr|^p\\\lesssim
\sum_{-R\le \tau\le\nu}(\nu-\tau+1)^{(\boldsymbol{\alpha}-1)p'} [v(Q_{0\tau})]^{1-p'}
\sum_{|l-\tau|\le m^*}\langle B_{m^*}(\cdot-\tau+\boldsymbol{\alpha}),
B_{m^*}(\cdot-l+\boldsymbol{\alpha})\\\le
\sum_{-R\le \tau\le\nu}(\nu-\tau+1)^{(\boldsymbol{\alpha}-1)p'} [v(Q_{0\tau})]^{1-p'}.
\end{multline*}
Combining this with \eqref{4T1more+BEST2} and \eqref{dopka12}, we obtain from
\eqref{dopka11} that
$$\mathfrak{C}_{\mathfrak{B}_{m^*}}\gtrsim
\sup_{\nu\in\mathbb{Z}}
\biggl(\sum_{\tau\ge\nu}
u(Q_{0\tau})\biggr)^{\frac{1}{p}}\biggl(\sum_{-R\le \tau\le\nu} (\nu-\tau+1)^{p'(\boldsymbol{\alpha}-1)}
[v(Q_{0\tau})]^{1-p'}\biggr)^{\frac{1}{p'}},$$
from which, tending $R\to\infty$, we extract the estimate $\mathfrak{C}\ge \mathfrak{C}_{\mathfrak{B}_{m^*}}
\gtrsim\mathbf{M}_+^{\boldsymbol{\alpha}}(0)$
for the best constant
of the inequality \eqref{4T1more+BEST}. Similarly, it can be shown that
$\mathfrak{C}\gtrsim\mathbf{M}_+^{\boldsymbol{\alpha}}(1)$, where
$$\mathbf{M}_+^{\boldsymbol{\alpha}}(1)=
\sup_{\tau\in\mathbb{Z}}
\biggl(\sum_{r\ge\tau}(r-\tau+1)^{p(\boldsymbol{\alpha}-1)}
u(Q_{0r})\biggr)^{\frac{1}{p}}\biggl(\sum_{r\le\tau} 
[v(Q_{0r})]^{1-p'}\biggr)^{\frac{1}{p'}}.$$ To do this, we need to
define a test function dual to $f^*_R$
\begin{equation*}%\label{function*}
g^*_R(y)=\sum_{\nu\le\tau\le R}(\tau-\nu+1)^{(\boldsymbol{\alpha}-1)(p-1)}
[u(Q_{0\tau})]B_{m^*,\tau-\boldsymbol{\alpha}}(y)
\end{equation*}
and substitute it into the dual to \eqref{dopka11-0} inequality of the form
\begin{equation*}\biggl(\sum_{\tau\in\mathbb{Z}}[v(Q_{0\tau})]^{1-p'}\biggl|\sum_{l\ge \tau} 
A_{l-\tau}(\boldsymbol{\alpha})\langle g, B_{m^*, l-\boldsymbol{\alpha}}\rangle
\biggr|^{p'}\biggr)^{1/p'}\lesssim
\mathfrak{C}_{\mathfrak{B}_{m^*}} 
\biggl(\sum_{l\in\mathbb{Z}}[u(Q_{0l})]^{1-p'}
\bigl|\langle g,
B_{m^*,l-\boldsymbol{\alpha}}\rangle\bigr|^{p'}\biggr)^{1/{p'}}.
\end{equation*} %В отличие от предыдущей оценки, необходимо применить процедуру формирования итерационных разностей в правую сторону.
 
Now suppose the maximum on the right--hand side \eqref{snizu} is achieved at
$\sup_{d\in\mathbb{N}}\mathfrak{M}(d,\kappa^*;v,u)$. 
Let us show the necessity of such a characteristic. 
To do this, for some $d_0$ and $\tau_0$ we define 
\begin{multline}\label{fu}h^*_{\tau_0}(y)
=\mathbf{\Lambda}_{m^*}^{-1}
\mathbf{\Psi}_{m^*,\bar{\boldsymbol{a}},\bar{\boldsymbol{s}}; n^*(1),\boldsymbol{\alpha}(1)}
(2^{d_0-1}y-\tau_0)\\=
\sum_{k=0}^{2\boldsymbol{\alpha}}(-1)^k
\binom{2\boldsymbol{\alpha}}{k}h_{\tau_0;k}(y)
=\sum_{|\varkappa|\le n^*}\lambda_{|\varkappa|}(n^*)(-1)^{|\varkappa|}
\sum_{k=0}^{2\boldsymbol{\alpha}}(-1)^k
\binom{2\boldsymbol{\alpha}}{k}h_{\tau_0;k,\varkappa}(y)\end{multline} with
$\bar{\boldsymbol{s}}\in\mathbb{R}$, which we will specified later, and
\begin{gather*}h_{{\tau_0};k}(y)=
\mathbf{\Psi}_{m^*,\bar{\boldsymbol{a}},\bar{\boldsymbol{s}}; n^*(1),\boldsymbol{\alpha}(0)}
(2^{d_0-1}y-\tau_0-k/2+\boldsymbol{\alpha}/2),\\
h_{{\tau_0};k,\varkappa}(y)=
\mathbf{\Psi}_{m^*,\bar{\boldsymbol{a}},\bar{\boldsymbol{s}}}
(2^{d_0-1}y-\tau_0-k/2-\varkappa/2+\boldsymbol{\alpha}/2).
\end{gather*}
We will need the representation \eqref{rysha+} for $f=h^*_{\tau_0}$
and \eqref{rysha+R} for $f=h_{{\tau_0};k,\varkappa}$ with 
$k=0,\ldots,2\boldsymbol{\alpha}$ and
$\varkappa=-n^*,\ldots,0,\ldots,n^*$.
In this situation, we fix $n^*$ and $m^*$ so that $n^*=m^*+\boldsymbol{\alpha}$.

Then (see \eqref{diff} and \eqref{dopka24})
\begin{multline}\label{fuI}\gamma_{m^*}
I_+^{\boldsymbol{\alpha}}h^*_{\tau_0}(x)=2^{-\boldsymbol{\alpha}d_0}\gamma_{n^*}
\mathbf{\Psi}_{n^*,\bar{\boldsymbol{a}},\bar{\boldsymbol{s}}; m^*(1),\boldsymbol{\alpha}(0)}
(2^{d_0-1}\cdot-\tau_0)=2^{-\boldsymbol{\alpha}d_0}
\sum_{|r|\le m^*}\lambda_{|r|}(m^*)(-1)^{|r|}
\\\times\sum_{|\kappa|\le n^*}\lambda_{|\kappa|}(n^*)(-1)^{|\kappa|}
\sum_{j= 0}^{n^*+1}(-1)^j\binom{1+n^*}{j}\ 
B_{n^*}\bigl(2^{d_0}(x-\bar{\boldsymbol{s}}-\bar{\boldsymbol{a}})-2\tau_0
+n^*-j-r-\kappa\bigr).
\end{multline} We have
${\rm supp}\,(I_+^{\boldsymbol{\alpha}}h^*_{\tau_0})=
\Bigl[\frac{2(\tau_0+\bar{\boldsymbol{s}}+
\bar{\boldsymbol{a}})-2n^*-m^*}{2^{d_0}},\frac{2(\tau_0+\bar{\boldsymbol{s}}+
\bar{\boldsymbol{a}})
+2n^*+m^*+2}{2^{d_0}}
\Bigr]$ (see \S\,\ref{dopol}). On the other hand, the supports of the functions
$\Psi_{n^*,\widetilde{\boldsymbol{a}},0}(2^{d-1}\cdot-\tau)$ forming the coefficients $\widetilde{\boldsymbol{\lambda}}_{d\tau}
^{\langle \widetilde{\boldsymbol{a}}\rangle}$
in \eqref{rysha+}, coincide with $$\left[\frac{2\tau
+2\widetilde{\boldsymbol{a}}-2n^*}{2^d}, 
\frac{2\tau+2\widetilde{\boldsymbol{a}}+2n^*+2}{2^d}\right]$$
(see\S\,\ref{dopol}), besides
\begin{multline}\label{12-11}
\gamma_{n^*}\mathbf{\Psi}_{n^*,\widetilde{\boldsymbol{a}},0}
(2^{d-1}\cdot-\tau)
=\sum_{|l|\le n^*}\lambda_{|l|}(n^*)(-1)^{|l|}
\\\times
\sum_{s= 0}^{n^*+1}(-1)^s\binom{1+n^*}{s}\ 
B_{n^*}\bigl(2^{d_0}(y-
\widetilde{\boldsymbol{a}})-2\tau+n^* -s-l
\bigr).\end{multline} 

Denote $\int_\mathbb{R}B_n(t) B_n(t-n)\,dt=:\mathbf{B}_n$
and observe that $\mathbf{B}_n>0$ for any $n\in\mathbb{N}$. 
We will call $\mathbf{B}_n$ as $\mathbf{B}_n-$intersection, which is minimal in some sense.
It is true: %с некоторым $\tau=\tau_0$
 (see \eqref{rysha+})
\begin{multline}\label{rass}\|I_{+}^{\boldsymbol{\alpha}}h^*_{\tau_0}\|^p_{B^{s,u}_{pq}(\mathbb{R})}
\gtrsim
2^{d_0 sp}|\widetilde{\boldsymbol{\lambda}}_{d_0\tau_0}
^{\langle\widetilde{\boldsymbol{a}}\rangle}|^p\,[u(Q_{(d_0-1)\tau_0})]=
2^{d_0 sp}2^{d_0p/2}|\langle 
I_{+}^{\boldsymbol{\alpha}}h^*_{\tau_0},
\widetilde{\mathbf{\Psi}}^{\langle\widetilde{\boldsymbol{a}}\rangle}_{(d_0-1)\tau_0}
\rangle|^p\,[u(Q_{(d_0-1)\tau_0})]\\=
2^{d_0 (s-\boldsymbol{\alpha})p}2^{d_0p}|\langle 
\mathbf{\Psi}_{n^*,\bar{\boldsymbol{a}},\bar{\boldsymbol{s}}; m^*(1),\boldsymbol{\alpha}(0)}
(2^{d_0-1}\cdot-\tau_0),
\mathbf{\Psi}_{n^*,\widetilde{\boldsymbol{a}},0}
(2^{d_0-1}\cdot-\tau_0)
\rangle|^p\,[u(Q_{(d_0-1)\tau_0})]\\=
2^{d_0 (s-\boldsymbol{\alpha})p}|\langle 
\mathbf{\Psi}_{n^*,\bar{\boldsymbol{a}},\bar{\boldsymbol{s}}; m^*(1),\boldsymbol{\alpha}(0)}
(\cdot/2-\tau_0),
\mathbf{\Psi}_{n^*,\widetilde{\boldsymbol{a}},0}
(\cdot/2-\tau_0)
\rangle|^p\,[u(Q_{(d_0-1)\tau_0})]
.\end{multline} 
If we choose $\bar{\boldsymbol{s}}=-1-m^*-4n^*$
and set $\bar{\boldsymbol{a}}=\widetilde{\boldsymbol{a}}=0$, the supports of
$\mathbf{\Psi}_{n^*,\bar{\boldsymbol{a}},\bar{\boldsymbol{s}}; m^*(1),\boldsymbol{\alpha}(0)}$
and $\mathbf{\Psi}_{n^*,\widetilde{\boldsymbol{a}},0}$
in \eqref{rass} intersect only at intersection of the supports
of the leftmost $B_{n^*}$ in the representation \eqref{12-11}, which is
$B_{n^*}(\cdot-2\tau_0+2n^*)$, and the rightmost $B_{n^*}$ in the representation
\eqref{fuI}, which corresponds to $B_{n^*}(\cdot-2\tau_0+3n^*)$. So, it will be
exactly intersection of the type $\mathbf{B}_{n^*}$.
The supports of other $B_{n^*}$ in the specified representations
will not intersect each other if $\bar{\boldsymbol{s}}=-1-m^*-4n^*$. 
Therefore, taking into account all the coefficients, we conclude that
\begin{equation}\label{18-12-2}
\langle 
\mathbf{\Psi}_{n^*,\bar{\boldsymbol{a}},\bar{\boldsymbol{s}}; m^*(1),\boldsymbol{\alpha}(0)}
(\cdot/2-\tau_0),
\mathbf{\Psi}_{n^*,\widetilde{\boldsymbol{a}},0}
(\cdot/2-\tau_0)
\rangle=\frac{16\mathbf{B}_{n^*}}{\gamma_{m^*}\gamma_{n^*}}=:\Theta(n^*).
\end{equation} From here and \eqref{rass} we obtain that
\begin{equation*}\label{rass1}\|I_{+}^{\boldsymbol{\alpha}}h^*_{\tau_0}\|_{B^{s,u}_{pq}(\mathbb{R})}
\gtrsim 
2^{d_0 (s-\boldsymbol{\alpha})}\,[u(Q_{(d_0-1)\tau_0})]^{\frac{1}{p}}\Theta(n^*).
\end{equation*}

On the other hand, 
$$\|h_{\tau_0}^*\|_{{B}_{pq}^{s+\kappa^*-\boldsymbol{\alpha},v}(\mathbb{R})}
\lesssim\sum_{|\varkappa|\le n^*}\lambda_{|\varkappa|}(n^*)
\sum_{k=1}^{2\boldsymbol{\alpha}}\binom{2\boldsymbol{\alpha}}{k}
\|h_{{\tau_0};k,\varkappa}\|_{{B}_{pq}^{s+\kappa^*-\boldsymbol{\alpha},v}(\mathbb{R})},
$$ where $h_{{\tau_0};k,\varkappa}$ are finite linear combinations of mutually 
orthogonal
integer shifts of $\psi_{m^*,\bar{\boldsymbol{a}},\bar{\boldsymbol{s}}}
(2^{d_0-1}\cdot-\tau_0-k/2
-\varkappa/2+\boldsymbol{\alpha}/2)$ with
$\bar{\boldsymbol{a}}$ and $\bar{\boldsymbol{s}}$ as above. For fixed $k$ and $\varkappa$, 
this corresponds to $\psi_{m^*,{\boldsymbol{a}},\boldsymbol{s}}(2^{d_0-1}\cdot-\tau_0)$
with
${\boldsymbol{a}}=(k+\varkappa-\boldsymbol{\alpha})/2-[(k+\varkappa-
\boldsymbol{\alpha})/2]$ and $\boldsymbol{s}=\bar{\boldsymbol{s}}+[(k+\varkappa-
\boldsymbol{\alpha})/2]$ (see \eqref{nosit1}).

From the representation \eqref{rysha+R} we can write for each $k+\varkappa$:
\begin{equation}\label{12-11-1}\|h_{\tau_0;k,\varkappa}\|_{{B}_{pq}^{s+\kappa^*-\boldsymbol{\alpha},v}(\mathbb{R})}\approx 
\Bigl\|\sum_{\tau\in\mathbb{Z}}
|{\boldsymbol{\lambda}}_{0\tau}^{\langle \boldsymbol{a}\rangle}|\chi_{Q_{0\tau}}\Bigr\|_{L^p(\mathbb{R},v)}+
\biggl(\sum_{d\in\mathbb{N}}2^{qd (s+\kappa^*-\boldsymbol{\alpha})}\Bigl\|\sum_{\tau\in\mathbb{Z}}
|{\boldsymbol{\lambda}}_{d\tau}^{\langle \boldsymbol{a}\rangle}|\chi_{Q_{(d-1)\tau}}
\Bigr\|_{L^p(\mathbb{R},v)}^q\biggr)^{1/q},\end{equation}
where the elements
${\boldsymbol{\lambda}}^{\langle \boldsymbol{a}\rangle}_{h_{\tau_0;k,\varkappa}}$
are given in \eqref{asty'25}
by $B_{m*,\boldsymbol{a}}$ and
$\Psi_{m^\ast,\boldsymbol{a},0}$ with
$\boldsymbol{a}$ as above. Due to the semi--orthogonality of the basis generated by
$B_{m*,\boldsymbol{a}}$ and
$\Psi_{m^\ast,\boldsymbol{a},0}$,
\begin{multline}\label{12-11-2}\|h_{\tau_0;k,\varkappa}\|_{{B}_{pq}^{s+\kappa^*-\boldsymbol{\alpha},v}
(\mathbb{R})}^p\approx
2^{d_0(s-\boldsymbol{\alpha}+\kappa^*+1)p}
\Bigl\|\sum_{\tau\in\mathbb{Z}}
|\langle h_{\tau_0;k,\varkappa}(\cdot),\Psi_{m^\ast,\boldsymbol{a},0}(2^{d_0-1}\cdot
-\tau)\rangle|
\chi_{Q_{(d_0-1)\tau}}
\Bigr\|_{L^p(\mathbb{R},v)}^p\\
=\mathbf{\Lambda}_{m^*}^{-2}2^{d_0(s-\boldsymbol{\alpha}+\kappa^*+1)p}
\sum_{\tau\in\mathbb{Z}}
\bigl|\langle \Psi_{m^\ast,{\boldsymbol{a}},{\boldsymbol{s}}}(2^{d_0-1}\cdot
-\tau_0),\Psi_{m^\ast,\boldsymbol{a},0}(2^{d_0-1}\cdot
-\tau)\rangle\bigr|^p[v(Q_{(d_0-1)\tau}].\end{multline}
Since the support of ${\Psi}_{m^*,\boldsymbol{a},0}$ is compact, then
the sum over $\tau$ is limited to $|\tau_0+\boldsymbol{s}-\tau|\le 2m^*$.
In this case, $v(Q_{(d_0-1)\tau}\lesssim v(Q_{(d_0-1)\tau_0}$ with a constant
depending on $n^*$ and $\boldsymbol{\alpha}$.
In view of
$\gamma_{m^*}{\Psi}_{m^*,\boldsymbol{a},0}(2^{d_0-1}\cdot-\tau)=
\sum_{|\rho|\le m^*}
\boldsymbol{\alpha}''_\rho\cdot
\psi_{m^*,\boldsymbol{a},0}(2^{d_0-1}\cdot-\rho-\tau)
$ we have (see \cite[\S\,4.1]{RMC}) 
$$\sum_{|\tau_0+\boldsymbol{s}-\tau|\le 2m^*}
\bigl|\langle \Psi_{m^\ast,{\boldsymbol{a}},{\boldsymbol{s}}}(2^{d_0-1}\cdot
-\tau_0),\Psi_{m^\ast,\boldsymbol{a},0}(2^{d_0-1}\cdot
-\tau)\rangle\bigr|\lesssim \mathbf{\Lambda}_{m^*}^{2}.$$
From this we obtain that $\mathfrak{C}\gtrsim 2^{-d_0\kappa^*}
\bigl[v(Q_{(d_0-1)\tau_0})\bigr]^{-\frac{1}{p}}
\bigl[u(Q_{(d_0-1)\tau_0})\bigr]^{\frac{1}{p}}$ for arbitrary $d_0$ and $\tau_0$.
Hence, 
\begin{equation}\label{18-12-3}
\mathfrak{C}\gtrsim \sup_{d\in\mathbb{N}}\mathfrak{M}(d,\kappa^*;v,u).
\end{equation}

The proof of statement (i) in the case of $I_{-}^{\boldsymbol{\alpha}}$ is similar.
In this case, it is necessary to apply the procedure for forming iterative differences 
to the right side.

(ii) {\it Sufficiency} of the conditions in this part follows from
Theorem \ref{S5T1} (iii) and the criteria for embeddings of Besov weighted spaces
(see \cite[Proposition 2.1]{HSc} and \cite[Proposition 5.1]{Ma}).

{\it Necessity} can be proven similarly to the estimate $\mathfrak{C}\gtrsim
\sup_{d\in\mathbb{N}}\mathfrak{M}(d,\kappa^*;v,u)$. Namely,
it is first assumed that the inequality \eqref{4T1more-_F} is valid.
Next, one should fix $n_\ast\in\mathbb{N}$ and $m_\ast\in\mathbb{N}$ satisfying
$n_*=m_\ast+\boldsymbol{\alpha}$ and the condition \eqref{condB}
with respect to parameters of the spaces $B^{s,u}_{pq}(\mathbb{R})$ and
$B^{s-\kappa_*-\boldsymbol{\alpha},w}_{pq}( \mathbb{R})$, respectively.
Then, by Theorem \ref{main'}, the representation \eqref{24_05} is true with
$\widetilde{\boldsymbol{\lambda}}^{\langle \widetilde{\boldsymbol{a}}\rangle}_{(I_{+}^
{\boldsymbol{\alpha}}f)}$ of the form \eqref{asty'++}, where
$\widetilde{\boldsymbol{\Phi}}^{\langle\widetilde{\boldsymbol{a}}\rangle}$
are defined by $B_{n*,\widetilde{\boldsymbol{a}}}$,
and $\widetilde{\boldsymbol{\Psi}}^{\langle \widetilde{\boldsymbol{a}}\rangle}$
(see \eqref{vazhnoo}) --- via
$\Psi_{n^\ast,\widetilde{\boldsymbol{a}},0}$, with
some $\widetilde{\boldsymbol{a}}\in
\{0,\pm1/2\}$. Thus, the representation \eqref{rysha+} takes place 
for the right--hand side of \eqref{4T1more-_F}.

Similarly, also based on Theorem \ref{main'},
$$
\|f\|_{{B}_{pq}^{s-\kappa_*-\boldsymbol{\alpha},w}(\mathbb{R})}\approx
\|{\boldsymbol{\lambda}}^{\langle \boldsymbol{a}\rangle}_f\|_{b_{pq}^{s-\kappa_*-\boldsymbol{\alpha}}(w)},$$
where the sequence ${\boldsymbol{\lambda}}^{\langle \boldsymbol{a}\rangle}_f$ with
$\boldsymbol{a}
\in\{0,\pm1/2\}$ is defined in \eqref{asty'25} with
${\boldsymbol{\Phi}}^{\langle \boldsymbol{a}\rangle}$ and
${\boldsymbol{\Psi}}^{\langle \boldsymbol{a}\rangle}$ (see \eqref{vazhnooW}),
specified by
$B_{m*,\boldsymbol{a}}$ and $\Psi_{m^\ast,\boldsymbol{a},0}$.
That is, for the left side of \eqref{4T1more-_F} the following is true:
\begin{equation}\label{rysha+R_*}\|f\|_{{B}_{pq}^{s-\kappa_*-\boldsymbol{\alpha},w}(\mathbb{R})}\approx 
\Bigl\|\sum_{\tau\in\mathbb{Z}}
|{\boldsymbol{\lambda}}_{0\tau}^{\langle \boldsymbol{a}\rangle}|\chi_{Q_{0\tau}}\Bigr\|_{L^p(\mathbb{R},w)}+
\biggl(\sum_{d\in\mathbb{N}}2^{qd (s-\kappa_*-\boldsymbol{\alpha})}\Bigl\|\sum_{\tau\in\mathbb{Z}}
|{\boldsymbol{\lambda}}_{d\tau}^{\langle \boldsymbol{a}\rangle}|\chi_{Q_{(d-1)\tau}}
\Bigr\|_{L^p(\mathbb{R},w)}^q\biggr)^{1/q}.\end{equation}

Next, for arbitrary $d_0$ and $\tau_0$, the function
$h_{*,\tau_0}=h^*_{\tau_0}$ is introduced into consideration (see \eqref{fu}). 
In this case (see \eqref{fuI}),
\begin{multline*}\gamma_{n^*}2^{\boldsymbol{\alpha}d_0}
I_+^{\boldsymbol{\alpha}}h^*_{\tau_0}(x)=
\mathbf{\Psi}_{n^*,\bar{\boldsymbol{a}},\bar{\boldsymbol{s}}; m^*(1),\boldsymbol{\alpha}(0)}
(2^{d_0-1}y-\tau_0)\\
=
\sum_{|l|\le m^*}\lambda_{|l|}(m^*)(-1)^{|l|}
\mathbf{\Psi}_{n^*,\bar{\boldsymbol{a}},\bar{\boldsymbol{s}}}
(2^{d_0-1}y-\tau_0-l/2)=:\sum_{|l|\le m^*}\lambda_{|l|}(m^*)(-1)^{|l|}
h_{*,t_0,l}(y).\end{multline*}
We will use the representation \eqref{rysha+R_*} for $f=h_{*,\tau_0}$
and \eqref{rysha+} for $h_{*,t_0,l}$ 
with $|l|\le m^*$ instead of $I_+^{\boldsymbol{\alpha}}f$.
First, we estimate the left--hand side of \eqref{4T1more-_F}.
To do this, we write (see \eqref{rysha+R_*})
\begin{multline*}\label{rassI}\|h_{*,\tau_0}\|^p_{B^{s-\kappa_*-\boldsymbol{\alpha},w}_{pq}
(\mathbb{R})} \gtrsim 
2^{d_0 (s-\kappa_*-\boldsymbol{\alpha})p}|{\boldsymbol{\lambda}}_{d_0\tau_0}
^{\langle{\boldsymbol{a}}\rangle}|^p\,[w(Q_{(d_0-1)\tau_0})]\\=
2^{d_0 (s-\kappa_*-\boldsymbol{\alpha})p}2^{d_0p/2}|\langle 
h_{*,\tau_0},
{\mathbf{\Psi}}^{\langle{\boldsymbol{a}}\rangle}_{(d_0-1)\tau_0}
\rangle|^p\,[w(Q_{(d_0-1)\tau_0})]\\=
2^{d_0 (s-\kappa_*-\boldsymbol{\alpha})p}2^{d_0p}|\langle 
\mathbf{\Psi}_{m^*,\bar{\boldsymbol{a}},\bar{\boldsymbol{s}}; n^*(1),\boldsymbol{\alpha}(1)}
(2^{d_0-1}\cdot-\tau_0),
\mathbf{\Psi}_{m^*,{\boldsymbol{a}},0}
(2^{d_0-1}\cdot-\tau_0)
\rangle|^p\,[w(Q_{(d_0-1)\tau_0})]
.\end{multline*} Just as in \eqref{rass} here it is sufficient to fix
$\bar{\boldsymbol{s}}$ in such a way to reduce the intersection of the supports
$\mathbf{\Psi}_{m^*,\bar{\boldsymbol{a}},\bar{\boldsymbol{s}}; n^*(1),\boldsymbol{\alpha}(1)}
(2^{d_0-1}\cdot-\tau_0)$ and $\mathbf{\Psi}_{m^*,{\boldsymbol{a}},0}$ with $\bar{\boldsymbol{a}}=
{\boldsymbol{a}}=0$
to $\mathbf{B}_{m^*}-$intersection of the supports of the rightmost $B_{m^*}$ in the corresponding representation
for $\mathbf{\Psi}_{m^*,\bar{\boldsymbol{a}},\bar{\boldsymbol{s}}; n^*(1),\boldsymbol{\alpha}(1)}
(2^{d_0-1}\cdot-\tau_0)$ and the leftmost $B_{m^*}$ in the representation for
$\mathbf{\Psi}_{m^*,{\boldsymbol{a}},0}$. This will lead us to the estimate
 \begin{equation}\label{12-11-3}
\|h_{*,\tau_0}\|_{B^{s-\kappa_*-\boldsymbol{\alpha},w}_{pq}
(\mathbb{R})}\gtrsim 2^{d_0 (s-\kappa_*-\boldsymbol{\alpha})}
\,[w(Q_{(d_0-1)\tau_0})]^{\frac{1}{p}}\Theta(m^*).
\end{equation} On the other hand, we will estimate
$\|I_+^{\boldsymbol{\alpha}}h_{*,\tau_0}\|^p_{B^{s,u}_{pq}
(\mathbb{R})}$ from above via $2^{(s-\boldsymbol{\alpha})d_0}\mathbf{\Lambda}_{n^*}^{2}$
by analogy with \eqref{12-11-1}
and \eqref{12-11-2}. Together with \eqref{12-11-3} this will lead us to the inequality
\begin{equation}\label{18-12-4}\mathds{C}\gtrsim 2^{-d_0\kappa_*}
\bigl[u(Q_{(d_0-1)\tau_0})\bigr]^{-\frac{1}{p}}
\bigl[w(Q_{(d_0-1)\tau_0})\bigr]^{\frac{1}{p}},\end{equation} and hence, 
to the required lower bound for the
best constant $\mathds{C}$ of the inequality \eqref{4T1more-_F}.
%In \eqref{18-12-3} and \eqref{18-12-4} the symbol $\gtrsim$ means dependence of 
%$\Theta(k)=\Theta(\mathbf{B}_k)$ (see \eqref{18-12-2}), where $\mathbf{B}_k$ decreases as
%$k\to\infty$ but $\mathbf{B}_k>0$ for all $k\in\mathbb{N}$.
\end{proof}

For the operators $I_{c^\pm}^{\boldsymbol{\alpha}}$ with $c\in\mathbb{R}$ 
the following analogous to Theorem \ref{S5T1-F} statement is true.

\begin{theorem}\label{S6T1-F}
Let 
$p> 1$, $0<q\le\infty$, $s\in\mathbb{R}$, 
weights $u,v\in\mathscr{A}_\infty^{\rm loc}$, 
$\boldsymbol{\alpha}\in\mathbb{N}$ 
and $f\in L_{\rm loc}(\mathbb{R})$. Suppose that
$f\equiv 0$ on $(-\infty,c)$ in the case of the operator $I_{c^+}^{\boldsymbol{\alpha}}$ 
and $f\equiv 0$ on $(c,\infty)$ for $I_{c^-}^{\boldsymbol{\alpha}}$.\\ 
{\rm (i)} $I_{c^\pm}^{\boldsymbol{\alpha}}$ is bounded from
${B}_{pq}^{s+\kappa^*-\boldsymbol{\alpha},v}(\mathbb{R})$ to ${B}_{pq}^{s,u}(\mathbb{R})$
with some $\kappa^*\in\mathbb{R}$ if and only if
$$
\widetilde{\mathfrak{C}}_{\pm}^{\boldsymbol{\alpha}}(\kappa^*,c)=
\widetilde{\mathbf{M}}_\pm^{\boldsymbol{\alpha}}(1,c)+
\widetilde{\mathbf{M}}_\pm^{\boldsymbol{\alpha}}(0,c)+
\sup_{d\in\mathbb{N}}\widetilde{\mathfrak{M}}_\pm(d,\kappa^*,c;v,u)<\infty,$$
where
\begin{equation*}
\widetilde{\mathfrak{M}}_\pm(d,\kappa,c;v,u):=
\sup_{\tau\in\pm\mathbb{N}_0} 2^{-d\kappa}
\bigl[v(Q_{(d-1)\tau}^{\langle c\rangle })\bigr]^{-\frac{1}{p}}
\bigl[u(Q_{(d-1)\tau}^{\langle c\rangle })\bigr]^{\frac{1}{p}}.\end{equation*}
Moreover, the norm $\|I_{c^\pm}^{\boldsymbol{\alpha}}\|_
{{B}_{pq}^{s+\kappa^*-\boldsymbol{\alpha},v}(\mathbb{R})\to {B}_{pq}^{s,u}(\mathbb{R})}$
of the operator $I_{c^\pm}^{\boldsymbol{\alpha}}$
is equivalent to $\widetilde{\mathfrak{C}}_{\pm}^{\boldsymbol{\alpha}}(\kappa^*,c)$.
\\
{\rm (ii)} Let $\kappa_*\in\mathbb{R}$ and assume that $0<\widetilde{\mathds{C}}$ --- 
the best independent of $f$ constant in 
\begin{equation}\label{4T1more-_F'}
\|f\|_{{B}_{pq}^{s-\kappa_*-\boldsymbol{\alpha},w}(\mathbb{R})}\lesssim 
\widetilde{\mathds{C}}\|I_{c^+}^{\boldsymbol{\alpha}}f\|_{{B}_{pq}^{s,u}(\mathbb{R})}.
\end{equation} The inequality \eqref{4T1more-_F'} holds if and only if 
$$\widetilde{\mathds{C}}(\kappa_*):=
\sup_{d\in\mathbb{N}_0}\widetilde{\mathfrak{M}}_\pm(d,\kappa_*,c;u,w)<\infty,$$
where
\begin{equation*}
\widetilde{\mathfrak{M}}_\pm(d,\kappa,c;u,w):=\sup_{\tau\in\pm\mathbb{N}_0}
\begin{cases}
\bigl[u(Q_{0\tau}^{\langle c\rangle })\bigr]^{-\frac{1}{p}}
\bigl[w(Q_{0\tau}^{\langle c\rangle })\bigr]^{\frac{1}{p}}, &d=0,\\
2^{-d\kappa}\bigl[u(Q_{(d-1)\tau}^{\langle c\rangle })\bigr]^{-\frac{1}{p}}
\bigl[w(Q_{(d-1)\tau}^{\langle c\rangle })\bigr]^{\frac{1}{p}}, & d\in\mathbb{N}.
\end{cases}\end{equation*} Besides, $\widetilde{\mathds{C}}\approx 
\widetilde{\mathds{C}}(\kappa_*)$.
\end{theorem}
\begin{proof}
Details of the proof can be found in \cite[\S\,6]{Ujms}. %(see also \cite{SMJ}).
\end{proof}

\begin{remark}{\rm
The characteristics found in Theorem \ref{S5T1} (i) differ from those presented in Theorem
\ref{S5T1-F} {\rm (i)} by components supremal in $d\in\mathbb{N}$, namely,
$$\sup_{d\in\mathbb{N}}\Bigl[\mathscr{M}_\pm^{\boldsymbol{\alpha}}(d,\kappa^*,1)
+\mathscr{M}_\pm^{\boldsymbol{\alpha}}(d,\kappa^*,0)\Bigr]\ge
\sup_{d\in\mathbb{N}}\mathfrak{M}(d,\kappa^*;v,u).$$
The characteristics obtained in Part {\rm (ii)} of Theorem \ref{S5T1} are also redundant
compared to the new ones established in Theorem \ref{S5T1-F}{\rm (ii)}: namely,
$\mathbb{C}_{\pm}^{\boldsymbol{\alpha}}(\kappa_*)\ge \mathds{C}(\kappa_*).$

The same is true relatively to the results of Theorems \ref{S6T1} and \ref{S6T1-F}}.
\end{remark}

At the end of Section, let us give an example.

\begin{example}\label{ex1}{\rm
Let $u(x)=w(x)=(1+|x|)^{-t}$ and $v(x)=(1+|x|)^{-t+sp}$ with $t>
\max\{1,p\boldsymbol{\alpha}-(p-1)\}$ and 
$\min\{1,\boldsymbol{\alpha}\}<\boldsymbol{\alpha}\le s$. For the operator 
$I_{0^+}^{\boldsymbol{\alpha}}$ from ${B}_{pq}^{s-\boldsymbol{\alpha},v}(\mathbb{R})$ 
to ${B}_{pq}^{s,u}(\mathbb{R})$ we have
$$\widetilde{\mathbf{M}}_+^{\boldsymbol{\alpha}}(0,0)\lesssim \begin{cases}
\sup_{\tau\in\mathbb{N}_0}(1+\tau)^{-\frac{t}{p}+\boldsymbol{\alpha}}, &
\boldsymbol{\alpha}p\le t< sp\\
\sup_{\tau\in\mathbb{N}_0}(1+\tau)^{-s+1}, & t\ge sp
\end{cases}$$ 
and
$$
\widetilde{\mathbf{M}}_+^{\boldsymbol{\alpha}}(1,0)\lesssim \begin{cases}
\sup_{\tau\in\mathbb{N}_0}(1+\tau)^{-\frac{t}{p}+\boldsymbol{\alpha}-\frac{1}{p'}}, &t< sp\\
\sup_{\tau\in\mathbb{N}_0}(1+\tau)^{-s+\boldsymbol{\alpha}}, & t\ge sp
\end{cases},$$ that is the functionals $\widetilde{\mathbf{M}}_+^{\boldsymbol{\alpha}}(i,0)$,
$i=0,1$, are finite. The condition
$$\widetilde{\mathfrak{M}}_+(d,0,0;v,u)<\infty$$
is also satisfied and $\widetilde{\mathfrak{M}}_+(d,0,0;u,w)<\infty$ (see also 
\cite{AM}). Therefore, under the conditions assumed and by virtue of Theorem \ref{S6T1-F}, 
the operator $I_{0^+}^{\boldsymbol{\alpha}}$ is bounded from
${B}_{pq}^{s-\boldsymbol{\alpha},v}(\mathbb{R})$ to ${B}_{pq}^{s,u}(\mathbb{R})$, besides, 
the ${{B}_{pq}^{s,u}(\mathbb{R})}\to {{B}_{pq}^{s-\boldsymbol{\alpha},w}(\mathbb{R})}$ 
norm inequality of the form \eqref{4T1more-_F'} holds.

For $I_{+}^{\boldsymbol{\alpha}}$, assume that $t+(p-1)+p(\boldsymbol{\alpha}-1)<sp$ 
and $\boldsymbol{\alpha}p<t$. Then
$${\mathbf{M}}_+^{\boldsymbol{\alpha}}(1)\lesssim \sup_{\tau\in\mathbb{N}}(1+\tau)^{-\frac{t}{p}+\boldsymbol{\alpha}-\frac{1}{p'}}
+\sup_{\tau\in-\mathbb{N}_0}(1-\tau)^{t/p-s+\boldsymbol{\alpha}-1/p}<\infty,$$
$${\mathbf{M}}_+^{\boldsymbol{\alpha}}(0)\lesssim \sup_{\tau\in\mathbb{N}}(1+\tau)^{-\frac{t}{p}+\boldsymbol{\alpha}}
+\sup_{\tau\in-\mathbb{N}_0}(1-\tau)^0<\infty.$$ Moreover,
${\mathfrak{M}}(d,0;v,u)<\infty$ and
${\mathfrak{M}}(d,0;u,w)<\infty$ for all $d\in\mathbb{N}_0$. Hence and by Theorem
\ref{S5T1-F},
the operator $I_{+}^{\boldsymbol{\alpha}}$ is bounded from
${B}_{pq}^{s-\boldsymbol{\alpha},v}(\mathbb{R})$ to ${B}_{pq}^{s,u}(\mathbb{R})$ and
the inequality \eqref{4T1more-_F} holds with $\kappa_*=0$.}
\end{example}

For the estimates in Example \ref{ex1}, the <<power rule>> for sums was applied 
(see \cite[Lemma 1]{BenGr2}).
%Возможен также переход от сумм к интегралам (см. \cite[замечание 5.7]{SMJ})

\section{Estimates under additional conditions on weights}
 
To simplify the quantitative characteristics obtained in Theorems \ref{S5T1-F} and
\ref{S6T1-F}, we can consider them on certain subclasses of weight functions.%, от которых они зависят.

\subsection{Weights with homogeneous anti--derivatives}

Parameters $\kappa_*$ and $\kappa^*$ in the definitions of function spaces
${B}_{pq}^{s+\kappa^*-\boldsymbol{\alpha},v}(\mathbb{R})$ and
${{B}_{pq}^{s-\kappa_*-\boldsymbol{\alpha},w}(\mathbb{R})}$ (see Theorems \ref{S5T1-F} and
\ref{S6T1-F}) depend on weight functions $v$ and $w$, as well as on weight $u$ of 
${B}_{pq}^{s,u}(\mathbb{R})$. That is, their values display themselves for specific 
$u$, $v$, and $w$ only.
In general case, $\kappa_*$ and $\kappa^*$
can be refined by narrowing the functionals $\sup_{d\in\mathbb{N}}\mathfrak{M}(d,\kappa^*;v,u)$
and $\sup_{d\in\mathbb{N}_0}\mathfrak{M}(d,\kappa_*;u,w)$ from Theorem \ref{S5T1-F}, 
as well as
$\sup_{d\in\mathbb{N}}\widetilde{\mathfrak{M}}_\pm(d,\kappa^*,c;v,u)$ and
$\sup_{d\in\mathbb{N}_0}\widetilde{\mathfrak{M}}_\pm(d,\kappa_*,c;u,w)$
to subclasses of weights.

Consider the collection of positively homogeneous weights.
Recall that a function $h\colon\mathbb{R}\to\mathbb{R}$ is positively homogeneous of degree
$s(h)$
if for any $x\in\mathbb{R}$ from the function's domain and any
$\lambda>0$
$$h(\lambda x)=\lambda^{s_h}\,h(x).$$ 
Rhe number $s_h$ is called the order of homogeneity.

In our situation, parameter $\lambda$ will take values  $2^{-d}$, $d\in\mathbb{N}$.
Given that the anti--derivative of each weight $\sigma_i$, $i=1,2$,
is a homogeneous function $\Sigma_i$ of degree $s_{\Sigma_i}$, one can give the functional
$\widetilde{\mathfrak{M}}_\pm(d,\kappa_*,c;\sigma_1,\sigma_2)$
a more transparent form in the case of $d\in\mathbb{N}$. Namely,
$$\widetilde{\mathfrak{M}}_\pm(d,\kappa,c;\sigma_1,\sigma_2)=
2^{-d\kappa}2^{\frac{d-1}{p}(s_{\Sigma_1}-s_{\Sigma_2})}
\widetilde{\mathfrak{M}}_\pm(0,\kappa,c;\sigma_1,\sigma_2).$$
Assuming $\kappa=(s_{\Sigma_1}-s_{\Sigma_2})/p$, we obtain that
$$\sup_{d\in\mathbb{N}_0}\widetilde{\mathfrak{M}}_\pm(d,\kappa,c;\sigma_1,\sigma_2)\simeq
\widetilde{\mathfrak{M}}_\pm(0,\kappa,c;\sigma_1,\sigma_2).$$
Examples are $\Sigma_i(x)\simeq |x|^\zeta$ на $|x|\ge |c|$ for 
$I_{c^+}^{\boldsymbol{\alpha}}$ with $c>0$ ore $I_{c^-}^{\boldsymbol{\alpha}}$ with $c<0$.

To simplify functionals of the type $\mathfrak{M}$, one can also use the direct 
\eqref{dpI'} and/or
inverse \eqref{dpII'} doubling conditions for Muckenhoupt--type functions in combination with 
Lemma \ref{basa}.  

\subsection{Using an average type condition}

The functionals $\mathbf{M}_\pm^{\boldsymbol{\alpha}}(\varepsilon)$,
as well as $\widetilde{\mathbf{M}}_\pm^{\boldsymbol{\alpha}}(\varepsilon,c)$, $\varepsilon=0,1$,
in assertions (i) of Theorems \ref{S5T1-F} and
\ref{S6T1-F} can be reduced to integral form if weights $\sigma$ involved  satisfy
the condition
\begin{equation}\label{usl}
\int_r^{r+1}\sigma\approx \sigma(\tau)\qquad \textrm{for some}\quad\tau\in[r,r+1]
\end{equation}
(see \cite{SMJ,UUaa}). The result when $\sigma=u$ and $\sigma=v$ are of the type 
\eqref{usl} is as follows:
\begin{gather*}
\mathbf{M}_\pm^{\theta}(\varepsilon)\approx
\sup_{\tau\in\mathbb{Z}}
\biggl(\int_{\tau}^\infty(x-\tau+1)^{p(\theta-1)\varepsilon}
u(x)\,dx\biggr)^{\frac{1}{p}}\biggl(\int_{-\infty}^\tau
(\tau-y+1)^{p'(\theta-1)(1-\varepsilon)}
[v(y)]^{1-p'}dy\biggr)^{\frac{1}{p'}},\\
\mathbf{M}_{-}^{\theta}(\varepsilon)\approx\sup_{\tau\in\mathbb{Z}}
\biggl(\int_{-\infty}^\tau (\tau-x+1)^{p(\theta-1)\varepsilon}
u(x)\,dx\biggr)^{\frac{1}{p}}\biggl(\int_{\tau}^\infty 
(y-\tau+1)^{p'(\theta-1)(1-\varepsilon)}
[v(y)]^{1-p'}dy\biggr)^{\frac{1}{p'}},\end{gather*}
\begin{align*}
\widetilde{\mathbf{M}}_\pm^{\theta}(\varepsilon)\approx&
\sup_{\tau\in\mathbb{N}_0}
\biggl(\int_{\tau+c}^\infty(x-\tau-c+1)^{p(\theta-1)\varepsilon}
u(x)\,dx\biggr)^{\frac{1}{p}}\\&\times\biggl(\int_{c}^{\tau+c}
(\tau+c-y+1)^{p'(\theta-1)(1-\varepsilon)}
[v(y)]^{1-p'}dy\biggr)^{\frac{1}{p'}},\end{align*}
\begin{align*}
\mathbf{M}_{-}^{\theta}(\varepsilon)\approx&\sup_{\tau\in-\mathbb{N}_0}
\biggl(\int_{-\infty}^{\tau+c} (\tau+c-x+1)^{p(\theta-1)\varepsilon}
u(x)\,dx\biggr)^{\frac{1}{p}}\\&\times \biggl(\int_{\tau+c}^c 
(y-\tau-c+1)^{p'(\theta-1)(1-\varepsilon)}
[v(y)]^{1-p'}dy\biggr)^{\frac{1}{p'}}.
\end{align*}

%\subsection{Смешанные случаи}

\section{Appendix}\label{dopol}
%\subsection{Системы сплайновых всплесков типа Баттла--Лемарье}
Here we give precise formulations of elements of the Battle--Lemari\'{e} type wavelet systems, 
in terms of which theorems of decomposition of Besov spaces
$B_{pq}^{s,w}(\mathbb{R})$ are formulated in \S\,\ref{decomp}.

We start by defining a basic spline of order $n\in\mathbb{N}_0$: namely,
$B_0=\chi_{[0,1)}$ and
\begin{equation*}
B_n(x):=(B_{n-1}\ast B_0)(x)=\int_0^1 B_{n-1}(x-t)\,dt.
\end{equation*}  The key property of $B_n$ in our case is the property of its derivative 
$B^{(k)}_{n}$ of order $k$:
\begin{equation}\label{diff}
B^{(k)}_{n}(\cdot)=\sum_{l=0}^k \binom {k}{l}B_{n-k}(\cdot-l)
\quad\quad (n\ge k\in\mathbb{N}).\end{equation}
Other properties of $B_n$ can be seen in \cite{Chui} and, for example, in \cite{Ujms}.

The $B_n$ is a scaling function, i.e. it generates in $L^2(\mathbb{R})$ 
the so--called multi--resolution analysis $KMA_{B_n}$ in
the form of expanding subspaces $V_d\subset V_{d+1}$, $d\in\mathbb{N}_0$, such that
$\mathrm{clos}_{L^2(\mathbb{R})}\Bigl(\bigcup_{d\in\mathbb{Z}} V_d\Bigr)=L^2(\mathbb{R})$ and
$\bigcap_{d\in\mathbb{Z}} V_d=\{0\}$. Moreover, for each $d\in\mathbb{N}_0$ the system
$$\bigl\{B_n(2^d\cdot-\tau)\bigr\}_{\tau\in\mathbb{Z}}$$ of integer shifts of 
dyadic dilations of B--spline $B_n$, which has compact support $[0,n+1]$,
forms a non--orthogonal Riesz basis in
$V_d\subset L^2(\mathbb{R})$.

Let $W_d$ denote the orthogonal complement of $V_d$ to $V_{d+1}$, i.e.
$V_{d+1}=V_d\oplus W_d$ for all $d\in\mathbb{N}$. Then $W_i\perp W_j$ for $i\not= j$.
From this, in particular, we obtain that 
$$L^2(\mathbb{R})=
V_{0}\oplus\Bigl(\bigoplus_{l\ge0} W_l\Bigr).$$ The spaces $W_d$, $d\in\mathbb{Z}$,
are called wavelet spaces. Just as the spaces $V_d$ are generated by integer
shifts of dyadic dilations
of the scaling function $B_n$, the spaces $W_d$ are generated by basis functions--wavelets,
constructed from
$B_n$. The multiresolution analysis $KMA_{\phi}$ with $\phi=B_n$ allows one to construct 
a wavelet Riesz basis in $L^2(\mathbb{R})$ using integer shifts of $B_n$ and the same type of 
shifts of dyadic dilations of the corresponding to
$B_n$ wavelets.
Within the framework of $KMA_{B_n}$, one can move from the scaling function $B_n$ to another, 
orthonormal one, satisfying the relation
\begin{equation}\label{phi_nn}
\hat{\phi}_n(\omega)=
\frac{\beta_n\ \hat{B}_n(\omega)}{(\mathrm{e}^{i\omega}r_1+1)\ldots
(\mathrm{e}^{i\omega}r_n+1)}=\beta_n\sum_{l_1=0}^\infty (-r_1\,\mathrm{e}^{i\omega})^{l_1}\ldots
\sum_{l_n=0}^\infty (-r_n\,\mathrm{e}^{i\omega})^{l_n}\hat{B}_n(\omega).\end{equation} 
Here $\beta_n:=(r_1+1)\ldots(r_n+1)$, where the sets $\{r_j\}_{j=1^n}\subset(0,1)$ are unique 
for each $n$, and $r_j$ are the roots of the Euler trigonometric polynomial (see, for example, 
\cite{UUjmaa}). The function $\phi_n$ is named after G. Battle \cite{B,B1} and P. G. 
Lemari\'{e}--Rieusset \cite{L}. 
One can write (see e.g. \cite{NovStech, UUjmaa, Ujms})
\begin{equation*}\label{phin}
\phi_n(x)=\beta_n\sum_{l_1=0}^\infty (-r_1)^{l_1}\ldots
\sum_{l_n=0}^\infty (-r_n)^{l_n}B_n(x+l_1\ldots + l_n).\end{equation*}
The Battle--Lemariet wavelets $\psi_n$ corresponding to $\phi_n$ satisfy the condition 
(see \cite{UUjmaa,Ujms})
\begin{align}\label{psihat}\hat{\psi}_{n,\boldsymbol{s}}(\omega)
=&\frac{\gamma_n}{2\mathrm{e}^{i\omega\boldsymbol{s}}}(1/r_1-\mathrm{e}^{-i\omega/2})\ldots
(1/r_n-\mathrm{e}^{-i\omega/2})\,
\sum_{k=0}^{n+1}\frac{(-1)^k(n+1)!}{k!(n+1-k)!}\mathrm{e}^{(n-k)i\omega/2}\nonumber\\
&\times \prod_{j=1}^n\sum_{m_j=0}^\infty (-r_j\,\mathrm{e}^{- i\omega})^{m_j}
\sum_{l_j=0}^\infty (-r_j\,\mathrm{e}^{i\omega/2})^{l_j}\ \hat{B}_n(\omega/2), \quad(\boldsymbol{s}\in\mathbb{Z})
\end{align} where 
${\gamma}_n=2^{-n}{\beta_n(r_1\ldots r_n)}$ and $\boldsymbol{s}$ is fixed.
Their integer shifts form an orthonormal basis in
$W_0\subset L^2(\mathbb{R})$. By analogy with
$\phi_n$, one can write:
\begin{multline*}\label{even}
\psi_{n,\boldsymbol{s}}(\cdot+\boldsymbol{s})={\gamma_n}
\sum_{m_1=0}^\infty(-r_1)^{m_1}\sum_{l_1= 0}^\infty (-r_1)^{l_1}\ldots 
\sum_{m_n= 0}^\infty(-r_n)^{m_n}\sum_{l_n= 0}^\infty(-r_n)^{l_n}\\
\Biggl[\sum_{k=0}^{n+1}\frac{(-1)^{k}(n+1)!}{k!(n+1-k)!}\ 
B_n(2\cdot-k- 2m_1+ l_1-\ldots - 2m_n+ l_n)\\
-\biggl\{\!\sum_{j_1\in\{1,\ldots,n\}}\frac{1}{r_{j_1}}\!\biggr\}
\sum_{k=0}^{n+1}\frac{(-1)^{k}(n+1)!}{k!(n+1-k)!}\ 
B_n(2\cdot -k+1- 2m_1+ l_1-\ldots - 2m_n+ l_n)+\\
\biggl\{\!\sum_{\substack{j_1,j_2\in\{1,\ldots,n\}\\ j_1\not=j_2}}
\!\!\!\frac{1}{r_{j_1}r_{j_2}}\!\biggr\}
\sum_{k=0}^{n+1}\frac{(-1)^{k}(n+1)!}{k!(n+1-k)!}\ 
B_n(2\cdot -k+2- 2m_1+ l_1-\ldots - 2m_n+ l_n)+\ldots\\
\ldots+\biggl\{\!\!\sum_{\substack{j_1,\ldots,j_c\in\{1,\ldots,n\}
\\ j_1\not=\ldots\not=j_c}}\!\!\!\frac{(-1)^c}{r_{j_1}\ldots r_{j_c}}\!\biggr\}
\sum_{k=0}^{n+1}\frac{(-1)^{k}(n+1)!}{k!(n+1-k)!}\ 
B_n(2\cdot-k+c- 2m_1+ l_1-\ldots - 2m_n+ l_n)+\ldots\\ +\ldots+
\frac{(-1)^n}{r_{1}\ldots r_{n}}\ 
\sum_{k=0}^{n+1}\frac{(-1)^{k}(n+1)!}{k!(n+1-k)!}\ 
B_n(2\cdot +(n-k)- 2m_1+ l_1-\ldots- 2m_n+ l_n)\Biggr]
.\end{multline*} 
The Fourier transform of ${\psi}_{n,\boldsymbol{s}}$ in \eqref{psihat} can be represented as
(see details in \cite[(2.31)]{Ujms})
\begin{align*}\hat{\psi}_{n,\boldsymbol{s}}(\omega)
=&\frac{\gamma_n}{2\mathrm{e}^{i\omega\boldsymbol{s}}}\,
\sum_{k=0}^{n+1}\frac{(-1)^k(n+1)!}{k!(n+1-k)!}\mathrm{e}^{(n-k)i\omega/2}\nonumber\\
&\times \prod_{j=1}^n\frac{|1-\mathrm{e}^{i\omega/2}r_j|^2}{r_j}\prod_{j=1}^n\sum_{m_j=0}^\infty (-r_j\,\mathrm{e}^{- i\omega})^{m_j}
\sum_{l_j=0}^\infty (r_j^2\,\mathrm{e}^{i\omega})^{l_j}\ \hat{B}_n(\omega/2).
\end{align*}

The orthonormal wavelet basis in $L^2(\mathbb{R})$ generated by the
pair $\{\phi_n,\psi_{n,\boldsymbol{s}}\}$ is formed as follows: it is
$\sqrt{2}\phi_{n}(x-\tau)$ and
$2^{d-1}\psi_{n,\boldsymbol{s}}(2^{d-1}x-\tau)$ for all
$\tau\in\mathbb{Z}$ and $d\in\mathbb{N}_0$.

Instead of $B_n$ one can take the function $B_{n,\boldsymbol{a}}(\cdot):=
B_n(\cdot-\boldsymbol{a})$ with
$\boldsymbol{a}\in\mathbb{R}$. This would form another $KMA_{B_{n,\boldsymbol{a}}}$ and, 
generally speaking, some other system $\{\phi_{n,\boldsymbol{a}},\psi_{n,\boldsymbol{a},\boldsymbol{s}}\}$,
generating an orthonormal basis of Battle--Lemariez wavelets in $L^2(\mathbb{R})$.

Since $\phi_{n,\boldsymbol{a}}$ and $\psi_{n,\boldsymbol{a},\boldsymbol{s}}$ have unbounded 
supports, in \cite{PSI}, with the aim of effective using spline systems of this type in 
decomposition theorems, a so-called localisation algorithm was proposed. 
It consists in forming, within the framework of
$KMA_{B_{n,\boldsymbol{a}}}$,
new basis elements which are finite linear combinations of
integer shifts of the scaling or wavelet Battle--Lemari\'{e} functions. 

As a localised analogue of $\phi_{n,\boldsymbol{a}}$ we will use the function
$\Phi_{n,\boldsymbol{a}}=B_{n,\boldsymbol{a}}$, which turns out to be a linear
combination of $n+1$
integer shifts of $\phi_{n,\boldsymbol{a}}$. Indeed, on the strength of  \eqref{phi_nn},
\begin{equation}\label{ss22'}
\hat{\phi}_{n,\boldsymbol{a}}(\omega)\prod_{j=1}^n\frac{1}{r_j+1}(\mathrm{e}^{i\omega}r_j+1)
=\hat{B}_{n,\boldsymbol{a}}(\omega)=:\hat\Phi_{n,\boldsymbol{a}}(\omega).\end{equation} 
Therefore, we have 
\begin{equation}\label{phiphi}
\textrm{supp}\,{\Phi}_{n,\boldsymbol{a}}=[\boldsymbol{a},\boldsymbol{a}+n+1]\qquad
\textrm{and}\qquad{\Phi}_{n,\boldsymbol{a}}(\cdot)=
\beta_n^{-1}\sum_{\varkappa=-n}^{0}\boldsymbol{\alpha}'_\varkappa\cdot 
\phi_{n,\boldsymbol{a}}(\cdot-\varkappa),\end{equation} where $\boldsymbol{\alpha}'_\varkappa>0$
for all $\varkappa=-n,\ldots,0$ and
$\sum_{\varkappa=-n}^{0}\boldsymbol{\alpha}'_\varkappa=\beta_n$.
Recall that integer shifts of ${\phi}_{n,\boldsymbol{a}}$ in \eqref{phiphi} are mutually orthogonal,
moreover, $0<r_j<1$ for all
$j=1,\ldots,n$ in \eqref{ss22'}. %Поэтому, $\{{\phi}_{n,\boldsymbol{a}}(\cdot-\tau)\}_{\tau\in\mathbb{Z}}$ --- базис Рисса в $V_0(B_{n,\boldsymbol{a}})$.

As a localised analogue of $\psi_{n,\boldsymbol{a},\boldsymbol{s}}$, we consider the
function
${\Psi}_{n,\boldsymbol{a},\boldsymbol{s}}$, satisfying the condition
\begin{multline}\label{obraz**}\hat{\Psi}_{n,\boldsymbol{a},\boldsymbol{s}}(\omega)=
\gamma_n^{-1}
\hat{\psi}_{n,\boldsymbol{a},\boldsymbol{s}}(\omega)\prod_{j=1}^n(1+\mathrm{e}^{-i\omega}r_j)
(1-\mathrm{e}^{i\omega}r_j^2)\\=
\frac{\hat{B}_{n,\boldsymbol{a}}(\omega/2)}{2\mathrm{e}^{i\omega\boldsymbol{s}}}
\sum_{k=0}^{n+1}\frac{(-1)^k(n+1)!}{k!(n+1-k)!}\mathrm{e}^{(n-k)i\omega/2}
\prod_{j=1}^n\frac{|1-\mathrm{e}^{i\omega/2}r_j|^2}{r_j}
\\=
\frac{\hat{B}_{n,\boldsymbol{a}}(\omega/2)}{2\mathrm{e}^{i\omega\boldsymbol{s}}}\sum_{k=0}^{n+1}
\frac{(-1)^k(n+1)!}{k!(n+1-k)!}\mathrm{e}^{(n-k)i\omega/2}
\prod_{j=1}^n[\rho_j-2\cos(j\omega/2))],\end{multline}
where $\rho_j:=r_j+1/r_j=2(2\alpha_j-1)$ with $\alpha_j>1$, $j=1,\ldots,n$ (see details in \cite{Ujms}). 
From this it is clear that $\Psi_{n,\boldsymbol{a},\boldsymbol{s}}$ is, on the one hand, 
a linear combination of
$2n+1$ mutually orthogonal integer shifts of ${\psi}_{n,\boldsymbol{s},\boldsymbol{a}}$.
On the other hand, $\Psi_{n,\boldsymbol{a},\boldsymbol{s}}$ is a finite linear combination of
integer and half shifts of $B_{n,\boldsymbol{a}}(2\cdot)$. 
Since all $r_j$ lie inside $(0,1)$,
the system $\bigl\{\Psi_{n,\boldsymbol{a},\boldsymbol{s}}
(\cdot-\tau)\bigr\}_{\tau\in\mathbb{Z}}$
forms a Riesz basis in $W_0(B_{n,\boldsymbol{a}})$. 
Also observe that 
\begin{equation}\label{nosit1}\textrm{supp}\,{\Psi}_{n,\boldsymbol{a},\boldsymbol{s}}=
[\boldsymbol{s}+\boldsymbol{a}-n,\boldsymbol{s}+\boldsymbol{a}+n+1]\quad\textrm{ and}\quad
{\Psi}_{n,\boldsymbol{a},\boldsymbol{s}}(\cdot)=\gamma_n^{-1}\sum_{|\varkappa|\le n}
\boldsymbol{\alpha}''_\varkappa\cdot \psi_{n,\boldsymbol{a},\boldsymbol{s}}(\cdot-\varkappa),\end{equation} 
where
\begin{equation*}%\label{positive0'}
\gamma_n^{-1}\sum_{|\varkappa|\le n}\boldsymbol{\alpha}''_\varkappa=2^n\prod_{j=1}^n
\frac{(1+r_j)(1-r_j^2)}{r_j(1+r_j)}=2^n\prod_{j=1}^n(1/r_j-r_j):=\mathbf{\Lambda}_n>0.
\end{equation*}
 
To formulate decomposition theorems, we will also need functions of the form:
\begin{equation*}\label{vsplesk}\hat{
{\Psi}}_{{n,\boldsymbol{a}},\boldsymbol{s};m(\Bbbk),\boldsymbol{\alpha}(\boldsymbol{\zeta})}(\omega)=
\hat{\Psi}_{n,\boldsymbol{a},\boldsymbol{s}}(\omega)
\biggl[\prod_{j=1}^{m}[\rho_j(m)-2\cos(\omega/2))]\biggr]^{\Bbbk}
\biggl[\sum_{\iota=0}^{2\boldsymbol{\alpha}}\frac{(-1)^\iota (2\boldsymbol{\alpha})!}{\iota!
(2\boldsymbol{\alpha}-\iota)!}\mathrm{e}^{(\boldsymbol{\alpha}-\iota)i\omega/2}
\biggr]^{\boldsymbol{\zeta}},
\end{equation*} where $\Bbbk,\boldsymbol{\zeta}\in\{0,1\}$.
Due to the fact that $r_j(m)[\rho_j(m)-2\cos(\omega/2)]=|1-\mathrm{e}^{i\omega/2}r_j(m)|^ 2$ 
and
$$\sum_{\iota=0}^{2\boldsymbol{\alpha}}\frac{(-1)^\iota (2\boldsymbol{\alpha})!}{\iota!
(2\boldsymbol{\alpha}-\iota)!}\mathrm{e}^{(\boldsymbol{\alpha}-\iota)i\omega/2}=
\mathrm{e}^{\boldsymbol{\ alpha}i\omega/2}(1-\mathrm{e}^{-i\omega/2})^{2\boldsymbol{\alpha}},$$
the element
$ {\Psi}_{{n, \boldsymbol{a}},\boldsymbol{s};m(1),\boldsymbol{\alpha}(\boldsymbol{\zeta})}$,
$\boldsymbol{\zeta}=0,1$, is linear combination of integer and half shifts of 
$\Psi_{n,\boldsymbol{a},\boldsymbol{s}}$. 
Observe that 
\begin{equation}\label{nosit2}\textrm{supp}\,
{\Psi}_{{n,\boldsymbol{a},\boldsymbol{s}};m(\Bbbk),
\boldsymbol{\alpha}(\boldsymbol{\zeta})}
=
[\boldsymbol{s}+\boldsymbol{a}
-n-m\Bbbk/2-\boldsymbol{\alpha}\boldsymbol{\zeta}/2,\boldsymbol{s}+\boldsymbol{a}+n
+1+ m\Bbbk/2+\boldsymbol{\alpha}\boldsymbol{\zeta}/2]\end{equation} and
\begin{equation}\label{ksi}{\Psi}_{{n,\boldsymbol{a},\boldsymbol{s}};m(1),\boldsymbol{\alpha}
(0)}(\cdot)=
\sum_{|\varkappa|\le m}\lambda_{|\varkappa|}(m)(-1)^{|\varkappa|} 
\Psi_{n,\boldsymbol{a},\boldsymbol{s}}(\cdot-\varkappa/2),\end{equation} 
\begin{multline}\label{ksi*}{\Psi}_{{n,\boldsymbol{a},\boldsymbol{s}};m(1),\boldsymbol{\alpha}
(1)}(\cdot)=
\sum_{|\varkappa|\le m}\lambda_{|\varkappa|}(m)(-1)^{|\varkappa|} 
\sum_{\zeta=0}^{2\boldsymbol{\alpha}}\binom{2\boldsymbol{\alpha}}{\zeta}
(-1)^{\zeta}
\Psi_{n,\boldsymbol{a},\boldsymbol{s}}(\cdot-\zeta/2+\boldsymbol{\alpha}/2-\varkappa/2)
\\\quad\textrm{с}\quad \sum_{|\varkappa|\le m}
\lambda_{|\varkappa|}(m)(-1)^{|\varkappa|}=\prod_{j=1}^m
\frac{1-r_j^2(m)}{r_j(m)}=2^{-m}\mathbf{\Lambda}_m,\end{multline}
(see \cite[p. 18]{Ujms}), %так как $r_j(m)[\rho_j(m)-2\cos(\omega/2)]=|1-{\rm e}^{i\omega/2}
%r_j(m)|^2$, 
besides,
\begin{equation}\label{lambda}\prod_{j=1}^m\bigl[(r_j+1/r_j)-2\cos(\omega/2)\bigr]=:
\sum_{j=0}^m(-1)^j\lambda_j\cos(j\omega/2),\end{equation}
where $\lambda_m=2$ and $0<\lambda_j=\lambda(\rho_1,\ldots,\rho_m)$ in the case $ j\not=m$, 
and $\lambda_j$ are even for $j\not=0$. For completeness, 
we will also write down that (see \eqref{obraz**})
\begin{equation}\label{formulll}{\mathbf{\Psi}}_{n,\boldsymbol{a},\boldsymbol{s}}(\cdot)=
\gamma_n
\sum_{|s|\le n}\lambda_{|s|}(n)(-1)^{|s|}
\sum_{k= 0}^{n+1}(-1)^k\binom{1+n}{k}\ 
B_{n}(2\cdot+(n-k)-s).\end{equation}

In conclusion of Section, for $\boldsymbol{a}\in\mathbb{R}$ 
we define pairs of functions of the form %поясним, как масштабирующая функция $\Phi_n$ и всплеск $\Psi_{n,m(\Bbbk)}$ генерируют базис во всем $L^2(\mathbb{R})$. Для этого обозначим 
\begin{gather}\label{vazhnoo} \widetilde{\mathbf{\Phi}}=\Phi_{n,\boldsymbol{a}}\qquad 
\textrm{and}\qquad 
\widetilde{\mathbf{\Psi}}^{\langle \Bbbk,\boldsymbol{\zeta}\rangle}=
\Psi_{{n,\boldsymbol{a}},\boldsymbol{s};m(\Bbbk),
\boldsymbol{\alpha}(\boldsymbol{\zeta})}/\mathbf{\Lambda}_{n},
\quad\textrm{where} \quad \Bbbk,\boldsymbol{\zeta}\in\{0,1\};
\\\label{vazhnooW} {\mathbf{\Phi}}^{\langle \boldsymbol{a}\rangle }=
\Phi_{n,\boldsymbol{a}}\qquad \textrm{and}\qquad {\mathbf{\Psi}}^{\langle \boldsymbol{a}
\rangle }=\Psi_{n,\boldsymbol{a},\boldsymbol{s}}/\mathbf{\Lambda}_{n}.\end{gather}

%Recall that integer shifts of $\Phi_{n,\boldsymbol{a}}$ generate a Riesz basis in 
%$V_0(B_{n,\boldsymbol{a}})$. Similarly, integer shifts of $\Psi_{n,\boldsymbol{ a}}$ form 
%a Riesz basis in $W_0(B_{n,\boldsymbol{a}})$. It holds that $\Phi_n=\Phi_{n,0}$ and 
%$\Psi_{n,m(0)}=\Psi_{n}=\Psi_{n,0}$.
In \eqref{vazhnoo} and \eqref{vazhnooW} the first components $\widetilde{\mathbf{\Phi}}$ and
${\mathbf{\Phi}}^{\langle \boldsymbol{a}\rangle }$ (the scaling functions or the
father wavelets) are identical to each other. 
We denote them differently only to distinguish the wavelet systems 
$\{\widetilde{\mathbf{\Phi}},
\widetilde{\mathbf{\Psi}}^{\langle \Bbbk,\boldsymbol{\zeta}\rangle}\}$ 
and $\{{\mathbf{\Phi}}^{\langle \boldsymbol{a}\rangle }, 
{\mathbf{\Psi}}^{\langle \boldsymbol{a}\rangle }\}$ from each other. 
The difference in the systems is in the second components
$\widetilde{\mathbf{\Psi}}^{\langle \Bbbk,\boldsymbol{\zeta}\rangle}$ and ${\mathbf{\Psi}}^{\langle \boldsymbol{a}
\rangle }$ or the wavelet functions (the mother wavelets). Here
$\widetilde{\mathbf{\Psi}}^{\langle \Bbbk,\boldsymbol{\zeta}\rangle}$ is a more general function, which, however, coincides with ${\mathbf{\Psi}}^{\langle \boldsymbol{a}
\rangle }$ when $\Bbbk=\boldsymbol{\zeta}=0$.

The smoothness of the functions in \eqref{vazhnoo} and \eqref{vazhnooW} is $n$. 
For $x\in\mathbb{R}$ we denote: \begin{equation}\label{ForRepr'}
\widetilde{\mathbf{\Phi}}_{\tau}(x):=\widetilde{\mathbf{\Phi}}(x-\tau)\quad\textrm{and}\quad
\widetilde{\mathbf{\Psi}}_{d\tau}^{\langle \Bbbk,\boldsymbol{\zeta}\rangle}(x):=2^{d/2}
\widetilde{\mathbf{\Psi}}^{\langle \Bbbk,\boldsymbol{\zeta}\rangle}(2^dx-\tau) \quad
(\tau\in\mathbb{Z},\, d\in\mathbb{N}_0),\end{equation} and also,
\begin{equation}\label{ForRepr'W}{\mathbf{\Phi}}_{\tau}^{\langle \boldsymbol{a}\rangle }
(x):={\mathbf{\Phi}}^{\langle \boldsymbol{a}\rangle }(x-\tau)\quad\textrm{and}\quad {\mathbf{\Psi}}_{d\tau}^{\langle \boldsymbol{a}\rangle }(x):=2^{d/2}{\mathbf{\Psi}}^{\langle \boldsymbol{a}\rangle }(2^dx-\tau) \quad
(\tau\in\mathbb{Z},\, d\in\mathbb{N}_0).\end{equation}
As noted in \cite[Remark 2.4]{Ujms}, the construction of ${\Psi}_{{n,\boldsymbol{a},
\boldsymbol{s}} ,m(\Bbbk),\boldsymbol{\alpha}(\boldsymbol{\zeta})}(\cdot-\tau)$,
$\Bbbk,\boldsymbol{\zeta}\in\{0,1\}$, starts with\\ $B_{n,2(\boldsymbol{s}+
\boldsymbol{a}+\tau)}(\cdot)$. As for observations \eqref{nosit1} and \eqref{nosit2}
concerning supports of ${\Psi}_{{n,\boldsymbol{a},\boldsymbol{s}}
,m(\Bbbk),\boldsymbol{\alpha}(\boldsymbol{\zeta})}$, they are slightly corrected in this paper
relatively to the information presented in \cite[\S\,2]{Ujms}.

The systems \eqref{ForRepr'W}, as well as \eqref{ForRepr'}, if they are generated
by elements \eqref{vazhnoo} with $\Bbbk=\boldsymbol{\zeta}=0$, form a semi--orthogonal 
Riesz basis in
$L^2(\mathbb{R} )$ (i.e. elements from $V_0$ are orthogonal to $W_l$ with $l\ge 0$, 
and functions of the space $W_{l_0}$ 
with fixed $l_0\in\mathbb{N}_0$
are orthogonal to $W_l$ if
$l\not=l_0$). The special structure of the elements of such systems allows them to be used
for obtaining lower bounds (or proofs of necessity) in decomposition theorems
(see \S\,\ref{decomp}). Upper bounds in such theorems (or proofs of sufficiency)
are less demanding and allow representation with respect to combined elements
of the type \eqref{ForRepr'}, when they
are generated by functions \eqref{vazhnoo} with $\Bbbk=1$ and/or $\boldsymbol{\zeta}=1$.

\medskip
{\bf Acknowledgment.} 
The research presented in § 4 of the paper was performed at 
 Steklov Mathematical Institute of Russian
Academy of Sciences under support of the RSF grant 24--11--00170, 
https://rscf.ru/project/24-11-00170/. The remaining work was done in framework of 
the state assignments of Ministry of Science and Higher Education of Russia for 
V. A. Trapeznikov Institute of Control Sciences of Russian
Academy of Sciences and Computing Center of Far Eastern Branch of
Russian
Academy of Sciences.

\renewcommand{\refname}{Bibliography}

\end{document}